\documentclass[a4paper,11pt]{amsart}

\makeatletter
\@namedef{subjclassname@2010}{%
  \textup{2010} Mathematics Subject Classification}
\makeatother
\usepackage{a4wide}

\usepackage{eulervm}

\usepackage{amscd}
\usepackage{amsmath}
\usepackage{amssymb}
\usepackage{amsthm}
\usepackage{float}
\usepackage[dvips]{graphicx}
\usepackage{enumerate}

\usepackage{array}
\usepackage{amscd}
\usepackage[all]{xy}
\usepackage{graphicx}
\usepackage{url}
\usepackage{enumitem}

\usepackage{amsfonts}
\usepackage{amssymb}
\usepackage{amsmath}
\usepackage{amsthm}

\usepackage{mathdots}

\usepackage[all]{xy}
\usepackage{graphicx}
\usepackage{mathrsfs}
\usepackage{color}
\usepackage{url}
\usepackage{array}
\usepackage{pifont}
\usepackage{tikz}
\usepackage{caption}
\usepackage{booktabs}
\usepackage{setspace}

\usetikzlibrary{datavisualization}
\usetikzlibrary{datavisualization.formats.functions}

\DeclareFontFamily{U}{rsfs}{%
\skewchar\font127}
\DeclareFontShape{U}{rsfs}{m}{n}{%
<-6>rsfs5<6-8.5>rsfs7<8.5->rsfs10}{}
\DeclareSymbolFont{rsfs}{U}{rsfs}{m}{n}
\DeclareSymbolFontAlphabet
{\mathrsfs}{rsfs}
\DeclareRobustCommand*\rsfs{%
\@fontswitch\relax\mathrsfs}

\theoremstyle{plain}
\newtheorem{thm}{Theorem}[section]
\newtheorem{prop}[thm]{Proposition}
\newtheorem{lem}[thm]{Lemma}
\newtheorem{defi}[thm]{Definition}
\newtheorem{rmk}[thm]{Remark}
\newtheorem{aside}[thm]{Aside}
\newtheorem{note}[thm]{Note}
\newtheorem{nota}[thm]{Notation}

\newtheorem{cor}[thm]{Corollary}

\newtheorem{prop-defi}[thm]{Proposition-Definition}
\newtheorem{thm-defi}[thm]{Theorem-Definition}
\newtheorem{lem-defi}[thm]{Lemma-Definition}

\newtheorem{conj}[thm]{Conjecture}
\newtheorem{exam}[thm]{Example}

\newdimen\argwidth
\def\db[#1\db]{
 \setbox0=\hbox{$#1$}\argwidth=\wd0
 \setbox0=\hbox{$\left[\box0\right]$}
  \advance\argwidth by -\wd0
 \left[\
 n.3\argwidth\box0 \kern.3\argwidth\right]}

\newcommand{\DelH}{\overline{\Delta}_H} 
\newcommand{\Del}{\overline{\Delta}}
\newcommand{\tal}{\widetilde{\alpha}}
\newcommand{\tbe}{\widetilde{\beta}}
\newcommand{\hbe}{\widehat{\beta}}
\newcommand{\hal}{\widehat{\alpha}}
\newcommand{\oA}{\overline{A}}

\newcommand{\SX}{\scriptscriptstyle  X}

\newcommand{\calExt}{\operatorname{\mathcal{E}\textit{xt}}}

\newcommand{\calHom}{\operatorname{\mathcal{H}\textit{om}}}

\newcommand{\obe}{\overline{\beta}}
\newcommand{\vep}{\varepsilon}
\newcommand{\vt}{\vartheta}
\newcommand{\aA}{\mathcal{A}}
\newcommand{\bB}{\mathcal{B}}

\newcommand{\CC}{\mathbb{C}}
\newcommand{\dD}{\mathcal{D}}

\newcommand{\fF}{\mathcal{F}}

\newcommand{\hH}{\mathcal{H}}
\newcommand{\iI}{\mathcal{I}}

\newcommand{\oO}{\mathcal{O}}

\newcommand{\QQ}{\mathbb{Q}}

\newcommand{\RR}{\mathbb{R}}
\newcommand{\PP}{\mathbb{P}}
\newcommand{\sS}{\mathcal{S}}
\newcommand{\tT}{\mathcal{T}}

\newcommand{\ZZ}{\mathbb{Z}}

\newcommand{\Supp}{\mathop{\rm Supp}\nolimits}
\newcommand{\Hom}{\mathop{\rm Hom}\nolimits}

\newcommand{\dR}{\mathbf{R}}

\newcommand{\NS}{\mathop{\rm NS}\nolimits}

\newcommand{\ch}{\mathop{\rm ch}\nolimits}
\newcommand{\rk}{\mathop{\rm rk}\nolimits}
\newcommand{\td}{\mathop{\rm td}\nolimits}

\newcommand{\Coh}{\mathop{\rm Coh}\nolimits}

\newcommand{\cneq}{\mathrel{\raise.095ex\hbox{:}\mkern-4.2mu=}}
\newcommand{\eqcn}{\mathrel{=\mkern-4.5mu\raise.095ex\hbox{:}}}

\newcommand{\Ree}{\operatorname{Re}}

\newcommand{\HN}{\operatorname{HN}}

\usepackage{chngcntr}
\usepackage{lipsum} 

\begin{document}

\title[Stability conditions, Bogomolov-Gieseker type inequalities and Fano 3-folds]{Stability conditions, Bogomolov-Gieseker type inequalities and Fano 3-folds}

\date{\today}

\author{Dulip Piyaratne}

\address{Kavli Institute for the Physics and Mathematics of the Universe (WPI)\\ The University of Tokyo Institutes for Advanced Study \\ The University of Tokyo \\ Kashiwa \\ Chiba 277-8583 \\ Japan.}

\email{dulip.piyaratne@ipmu.jp}



\begin{abstract}
We develop a framework to   modify  the  Bogomolov-Gieseker type inequality conjecture introduced by Bayer-Macr\`i-Toda, in order to  construct a family of geometric Bridgeland stability conditions on any smooth projective 3-fold.
 We show that it is enough to check these modified inequalities on a small class of tilt stable objects.
 We extend some of the techniques in the works by Li and Bernardara-Macr\`i-Schmidt-Zhao to formulate a strong form of 
 Bogomolov-Gieseker  inequality for tilt stable objects on Fano 3-folds.
 Consequently, we  establish our modified Bogomolov-Gieseker type inequality conjecture for general Fano 3-folds, including an optimal inequality for the blow-up of $\PP^3$ at a point.
\end{abstract}

\maketitle

\section{Introduction}
\subsection{Motivation and background}
The notion of stability conditions on triangulated categories 
was introduced by Bridgeland (see \cite{Bristab}). 
Such a stability condition on a triangulated category is defined by giving a bounded t-structure together with a stability function on its heart satisfying the Harder-Narasimhan  property. 
This can be interpreted essentially as an abstraction of the usual slope stability for sheaves. 
Construction of Bridgeland stability conditions on the bounded derived category of a given projective threefold is an important problem. However, unlike for a projective surface, there is no known construction which gives stability conditions for all projective threefolds. See \cite{Huystab, MS} for further details.   

In \cite{BMT}, Bayer, Macr\`i and Toda introduced a   conjectural construction of Bridgeland stability conditions for any projective threefold.
Here the problem was reduced to proving so-called Bogomolov-Gieseker type inequality holds for certain tilt stable objects. 
It has been shown to hold for 
 all Fano 3-folds with Picard rank one (see \cite{BMT, Mac, Sch1, Li}), abelian 3-folds (see \cite{MP1, MP2, Piy1, Piy2, BMS}), \'etale quotients of abelian 3-folds (see \cite{BMS}), some toric 3-folds (see \cite{BMSZ}) and 3-folds which are  products of projective spaces and abelian varieties (see \cite{Kos}).  
Recently, Schmidt found a counterexample to the original Bogomolov-Gieseker type inequality conjecture when $X$ is the blowup at a point of $\mathbb{P}^3$ (see \cite{Sch2}). 
Therefore, this inequality needs some modifications in general setting and this is one of the main goals of this paper.

\subsection{Modification of Bogomolov-Gieseker type inequality conjecture} 
Let $X$ be a smooth projective 3-fold, and let  $H \in \NS(X)$ and  $B \in \NS_{\mathbb{R}}(X)$ are some fixed classes such that $H$ is ample.
Let  $\beta \in \RR$ and $\alpha \in \RR_{>0}$.  
In  \cite{BMT}, the authors tilted the abelian category of coherent sheaves on $X$ with respect to a torsion pair coming from 
usual slope stability, 
 to get the abelian category $\bB_{H, B+ \beta H}$. 
 Moreover, they introduced the notion of tilt stability by defining the 
 $\nu_{H,B + \beta H, \alpha}$ tilt slope on  $\bB_{H, B+\beta H}$   by
$$
\nu_{H,B + \beta H, \alpha} : \bB_{H, B+\beta H} \ni E \mapsto \frac{H \ch_2^{B+\beta H}(E) - (\alpha^2/2)H^3 \ch_0(E)}{H^2\ch_1^{B +\beta H}(E)} \in \RR \cup \{+\infty\}.
$$
We modify the expression in the Bogomolov-Gieseker type  inequality conjecture by introducing 
an extra term 
 $\xi \in \mathbb{R}_{\ge 0}$ together with a class 
$\Lambda \in H^4(X, \QQ)$ satisfying
$\Lambda \cdot H =0$:
\begin{equation*}
D^{B, \xi}_{\alpha, \beta}(E) = \ch_3^{B + \beta H}(E) + \left(  \Lambda - \left( \xi+  \frac{1}{6}\alpha^2 \right) H^2 \right) \ch_1^{B + \beta H}(E).
\end{equation*}
More precisely,  we   conjecture the following for a family of stability parameters. 
Let $A : B + \RR \langle H \rangle  \to \RR_{\ge 0}$ be a continuous function.
\begin{conj}[=\,\ref{conj:BG-ineq}]
\label{conj:intro-BG-ineq}
There exist
$\Lambda \in H^4(X, \QQ)$ satisfying $\Lambda \cdot H =0$, and
a constant $\xi({A})  \in \RR_{\ge 0}$ such that 
for any   $(\beta, \alpha) \in \RR \times \RR_{>0}$,  
with $\alpha \ge A(B +\beta H)$,
all $\nu_{H, B+ \beta H, \alpha}$ tilt slope stable objects $E \in \bB_{H, B+ \beta H}$
with $\nu_{H, B+ \beta H, \alpha}(E) =0$ satisfy the inequality:
\begin{equation*}
D^{B,\xi(A)}_{\alpha, \beta}(E) \le 0.
\end{equation*}
Hence, for any $\xi \ge \xi (A)$, we have $D^{B,\xi}_{\alpha, \beta}(E) \le 0$.
\end{conj}

As similar to \cite{BMT, BMS}, the above   modification conjecturally gives us a family
of Bridgeland stability conditions. More specifically, 
when our modified Conjecture  holds for $X$, the  tilt of $\bB_{H,B+\beta H}$  as in the construction of \cite{BMT} together with
some central charge functions define those stability conditions. See Theorem \ref{thm:stab-family} for further  details.

Our modified conjectural inequality coincides with the  Bogomolov-Gieseker type inequality in \cite{BMT} 
when 
$A = 0$, $\Lambda =0$ and $\xi(A) =0$.
In this paper, we are mostly interested in the following choice:
\begin{equation}
\label{def:intro-Lambda}
\Lambda = \frac{c_2(X)}{12} - \frac{c_2(X)\cdot H}{12 H^3} H^2.
\end{equation}
Furthermore, this $\Lambda$ vanishes   for many 3-folds where the Bogomolov-Gieseker type inequality conjecture  in \cite{BMT} holds. 
For example,  when $X$ is an abelian 3-fold 
  ($c_2(X) =0$), or a Fano 3-fold with Picard rank one ($c_2(X)$ is proportional to $H^2$).

In Section \ref{subsec:betabar-restriction} we   extend  the notion of $\obe$-stability in \cite{Li, BMS}. 
For the continuous function $A : B + \RR \langle H \rangle  \to \RR_{\ge 0}$, 
we define $\obe_A(E) $ to be  the set of roots  $\obe$ of 
\begin{equation*}
\label{eqn:intro-betabar-relation}
H \ch_2^{B + \obe H}(E) - \frac{1}{2} (A(B + \obe H))^2 H^3 \ch_0(E) = 0.
\end{equation*}
We call an object   $E$ in  $D^b(X)$ is   $\obe_{A}$-stable if  for any $\obe \in \obe_A(E)$ there is an open neighbourhood $U \subset \mathbb{R}^2 $ containing $(\obe, A(B+ \obe H))$ such that for any $(\beta, \alpha) \in U$ with $\alpha>0$, 
$E \in \bB_{H, B+\beta H}$ is $\nu_{H, B + \beta H,  \alpha}$-stable. 

In this paper we reduce the requirement of our 
modified Bogomolov-Gieseker type inequalities to $\obe_A$ stable objects. 
More precisely, we show that Conjecture \ref{conj:intro-BG-ineq} is equivalent to the following. See Theorem \ref{thm:conj-equivalence} for further details. 
\begin{conj}[=\,\ref{conj:limitBG}]
\label{conj:intro-limitBG}
There exist
$\Lambda \in H^4(X, \QQ)$ satisfying $\Lambda \cdot H =0$, and
a constant $\xi({A})  \in \RR_{\ge 0}$ such that 
any $\obe_{A}$-stable object $E$ in  $D^b(X)$ satisfies the inequality:
$$
D^{B,\xi(A)}_{A(B +\obe H ) ,  \obe}(E) \le 0, \ \text{ for each } \obe \in \obe_A(E).
$$ 
\end{conj}

\subsection{Bridgeland stability conditions on Fano 3-folds}
In this paper we extend the works of \cite{Li, BMSZ} to establish Conjecture \ref{conj:intro-limitBG} for Fano 3-folds in optimal sense, that is having a minimal $\xi(A)$.

\begin{thm}[{\cite{Li}, Picard rank one case; \cite{BMSZ}, general Fano 3-folds; Theorem \ref{thm:fano3-A=0} for an optimal inequality on general Fano 3-folds}]
\label{thm:intro-Fano3-A=0}
Let $X$ be a smooth Fano 3-fold and $B$, $H$ are proportional to $-K_X$. 
Then  Conjecture  \ref{conj:intro-limitBG} holds on $X$ with respect to $A=0$ for some finite $\xi(A) \ge 0$ and $\Lambda$ as in \eqref{def:intro-Lambda}.
\end{thm}

Moreover, for the blow-up of $\PP^3$ at a point we show that an optimal modified Bogomolov-Gieseker type inequality holds. 
\begin{thm} [=\,\ref{thm:fano3-d-7}]
\label{thm:intro-blowup-P3-xi=0}
Let $X$ be the blow-up of $\PP^3$ at a point, and let $H = {-K_X}/{2}$.  
Let $A : \RR \langle H \rangle \to \RR$ be the continuous function defined by, for $\beta \in [-1/2, \, 1/2]$
$$
A(\beta H) = \begin{cases}
(1- \beta) & \text{if } \beta \in [-1/2, \, 0) \\
(1+ \beta) & \text{if } \beta \in [0, \, 1/2]
\end{cases}
$$
together with the relation $A((\beta + 1)H) = A(\beta H)$.
Then the modified Bogomolov-Gieseker type inequality in 
Conjecture \ref{conj:intro-limitBG} holds for $X$, with $\xi(A) =0$  and $\Lambda$ as in \eqref{def:intro-Lambda}. 
\end{thm}

The main ideas of the proofs of above results for Fano 3-folds 
are similar to the work of Li and Bernardara-Macr\`i-Schmidt-Zhao  in \cite{Li, BMSZ}.
 More precisely, by dualizing and tensoring by a line bundle, we reduce the requirement of the 
Bogomolov-Gieseker type inequalities to  $\obe_A$ stable objects having $\obe_A$ values in a unit interval in $\RR$ such that $\ch_0 \ge 0$. 
Then we compare the tilt slopes of such  objects with the tilt slopes of certain tilt stable line bundles and their shift by $[1]$. In this way we get certain $\Hom$ vanishings, and by using the Serre duality we obtain  a bound for the Euler characteristic involving our $\obe_A$ stable objects. Conclusively, the Hirzebruch-Riemann-Roch formula gives us the modified Bogomolov-Gieseker type inequalities for those restricted class of tilt stable objects. 
Particularly, the following strong form of 
Bogomolov-Gieseker  inequality for  tilt stable objects on Fano 3-folds, is crucial for us. 
This generalizes the earliest results \cite[Proposition 3.2]{Li} and \cite[Theorem 3.1]{BMSZ}, and see Remark \ref{rem:relation-with-BMSZ} for further details. 
\begin{thm}[=\,\ref{thm:strong-BG-ineq}]
\label{thm:intro-strong-BG-ineq}
Let $X$ be a Fano 3-fold and let $H$ and $B$ be classes proportional to $-K_X$.
Let $E$ be a tilt stable object with finite tilt slope and non-isomorphic to 
 $\oO_X(mH)[1]$ or $\iI_Z(mH)$ for any  $m \in \ZZ$ and $0$-subscheme $Z \subset X$.
Further we assume $\ch_0(E) \ne 0$, and $E$ satisfies certain conditions, namely \eqref{eqn:strong-BG-ineq-condition}, 
\eqref{eqn:strong-BG-ineq-condition-2}, and \eqref{eq:stron-BG-condition3}. 
Then 
$$
\frac{\overline{\Delta}_{H}(E)}{(H^3\ch_0(E))^2} \ge \kappa(X).
$$
Here $\Del_H$ is the discriminant as  in Definition \ref{def:discriminant}, and $\kappa(X)$ is a constant as in Definition \ref{def:kappa(X)}.
\end{thm}

\subsection{Relation to the existing works}
This paper supersedes the author's unpublished work \cite{PiyFano3}. 

A modification  of the Bogomolov-Gieseker type  inequality conjecture for Fano 3-folds appeared 
in \cite{BMSZ} and the author's unpublished preprint \cite{PiyFano3} almost at the same time.
One of the key ingredients in those works is the formulation of a strong form of Bogomolov-Gieseker  inequality 
for tilt stable objects on Fano 3-folds having  higher Picard ranks, extending the previous work of Li  in the Picard rank one case (see \cite{Li}). However, both formulations had a gap and it was fixed by the authors in \cite{BMSZ}. 
In this paper we further strengthen their inequality as presented in Theorem \ref{thm:intro-strong-BG-ineq} above.
See Remark \ref{rem:relation-with-BMSZ} for further details. 

Following similar ideas in \cite{Li, BMSZ}, we 
 establish Conjecture \ref{conj:intro-limitBG} or equivalently Conjecture \ref{conj:intro-BG-ineq}  for Fano 3-folds as stated in Theorems \ref{thm:intro-Fano3-A=0} and \ref{thm:intro-blowup-P3-xi=0}. 
The class  $\Gamma$  in Theorem 1.1 of \cite{BMSZ} is exactly equal to 
the class $\left( -\Lambda + \xi(A)H^2\right)$ with $A=0$ in our paper.
However, the modified Bogomolov-Gieseker type inequality 
that we propose in this paper is rather general, and also, the following are some significant points relevant to Fano 3-folds. 
\begin{enumerate}
 \item According to our notation, in \cite{BMSZ} the authors  only consider the modified Bogomolov-Gieseker type inequalities for Fano 3-folds with respect to $A =0$. 
 
\item  Since Theorem \ref{thm:intro-strong-BG-ineq} further generalizes \cite[Theorem 3.1]{BMSZ}, our modified 
 Bogomolov-Gieseker type inequalities for Fano 3-folds become stronger, specifically for the Fano 3-folds having 
 $\xi(A) >0$ in Theorem \ref{thm:intro-Fano3-A=0}. 
 
 \item In \cite{BMSZ}, the authors did not optimize the modified Bogomolov-Gieseker type inequalities for the blow-up of $\PP^3$ at a point. Specifically, $\Gamma$ in  \cite[Theorem 1.1]{BMSZ} is a class with $\Gamma \cdot H >0$ (see Sections 4.B and 6 of \cite{BMSZ}). However, in Section \ref{sec:blowup-P3}, we show that $\xi(A) =0$ for some non-zero $A$ as stated in Theorem \ref{thm:intro-blowup-P3-xi=0} above.
\end{enumerate}
 
\subsection{Plan of the paper} 
In Section \ref{sec:prelim}, we briefly recall the notion of tilt stability and some  important results associated to sheaves and Fano 3-folds. In Section \ref{sec:tiltproperties}, we discuss properties of tilt stable objects in detail. In particular, we see that tilt stability on 3-folds is preserved under the dualizing of objects. More precisely, we see that objects in the first tilted category of 3-folds  behave somewhat similar to coherent sheaves on a projective surface under the dualizing. 
In Section \ref{sec:BGconj}, we  develop the framework to   modify  the Bogomolov-Gieseker type inequality introduced by Bayer, Macr\`i and Toda, in order to  construct a family of geometric Bridgeland stability conditions on any smooth projective 3-fold. Moreover, we  introduce the notion of $\obe_A$ stability, and reduce the requirement of 
modified Bogomolov-Gieseker type inequality conjecture to those stable objects.
In Section \ref{sec:Hom-vanishing}, we  get certain $\Hom$ vanishing results for $\obe_A$ stable objects with respect to some line bundles. We prove a strong from of Bogomolov-Gieseker  inequality for Fano 3-folds in Section \ref{sec:strong-BG-ineq-for-tiltstable}. In Section \ref{sec:Fano3-A=0}, we establish the modified Bogomolov-Gieseker type inequalities for Fano 3-folds, and in Section \ref{sec:blowup-P3} an optimal inequality for the blow-up of $\PP^3$ at a point.


\subsection{Notation}
Let us collect some of the important notations that we use in this paper as follows:

\begin{itemize}[leftmargin=*]

\item When  $\aA$ is the heart of a bounded t-structure  on 
a triangulated category $\dD$,  by 
$H_{\aA}^i(-)$ we denote the corresponding  $i$-th cohomology functor. 

\item 
For a set of objects $\sS \subset \dD$ in a triangulated category $\dD$,  by 
$\langle \sS \rangle \subset \dD$ we denote
its extension closure, that is the smallest extension closed subcategory 
of $\dD$ which contains $\sS$.

\item Unless otherwise stated,  throughout this paper, all the varieties are smooth projective and defined over 
$\mathbb{C}$.  For a variety $X$, by $\Coh(X)$ we denote  the category of 
coherent sheaves on $X$, and by $D^b(X)$ we denote the bounded derived category of $\Coh(X)$.

\item For a variety $X$, 
by $\omega_X$ we denote its canonical line bundle, and let $K_X = c_1(\omega_X) $.

\item For  $ M = \QQ, \RR, \text{ or } \CC$  we write $\NS_M(X) = \NS(X) \otimes_{\ZZ} M$.

\item For $E, F \in D^b(X)$, denote $\hom_{\SX}(E,F) = \dim \Hom_{\SX}(E,F)$, and when $E$ is a sheaf, 
$h^i(E) = \dim H^i(E, X)$.  

\item For  the bounded derived category of a variety $X$, we simply write  $\hH^i(-)$ for $H_{\Coh(X)}^i(-)$.

\item For $0 \le i \le \dim X$, $\Coh_{\le i}(X)  = \{E \in \Coh(X): \dim \Supp(E)  \le i  \}$, $\Coh_{\ge i}(X) = \{E \in \Coh(X): \text{for } 0 \ne F \subset E,   \ \dim \Supp(F)  \ge i  \}$ and $\Coh_{i}(X) = \Coh_{\le i}(X) \cap \Coh_{\ge i}(X)$. 

\item For $E \in D^b(X)$, $E^\vee = \dR \calHom(E, \oO_X)$. When $E$ is a sheaf we write its dual sheaf $\hH^0(E^\vee)$ by $E^*$.

\item The skyscraper sheaf of a  closed point $x\in X$ is  denoted by $\oO_x$.

\item For  $B \in \NS_{\RR}(X)$, the twisted Chern character $ch^B(-) = e^{-B} \cdot \ch(-)$. 
For ample $H \in \NS(X)$ and, $\mu_{H,B}(E) = (H^2 \ch_1^B(E))/(H^3 \ch_0(E))$.
We write $\mu_H = \mu_{H,0}$ and $\mu_{H, \beta} = \mu_{H, \beta H}$.

\item  Tilt slope on $\bB_{H,B}$ is defined by $\nu_{H,B, \alpha} (E) = \frac{H \ch_2^{B}(E) - (\alpha^2/2) H^3\ch_0(E)}{H^2 \ch^B_1(E)}$. Sometimes we write $\nu_{\beta, \alpha} = \nu_{H, \beta H, \alpha}$ and 
$\bB_{\beta} = \bB_{H, \beta H}$. 

\item $\HN^{\mu}_{H, B}(I) = \langle E \in \Coh(X) : E \text{ is } \mu_{H , B}\text{-semistable with }
\mu_{H , B}(E) \in I \rangle$. Similarly, we define  $\HN^{\nu}_{H, B}(I) \subset \bB_{H,B}$. 

\item For $E \in \bB_{H,B}$ we write
$E^i = H^{i}_{\bB_{H,-B}}(E^\vee)$. 
So for example $E^{12} =H^{2}_{\bB_{H, B}} \left( \left(H^{1}_{\bB_{H,-B}}(E^\vee)\right)^\vee\right)$.

\end{itemize}

\subsection{Acknowledgements}
I am grateful to Sergey Galkin, 
Chen Jiang, Ilya Karzhemanov, and Alexander Kuznetsov for some useful discussions on Fano varieties.
Special thanks go to 
Marcello Bernardara, Emanuele Macr\`i, Benjamin Schmidt, and  Xiaolei Zhao for drawing my attention to the problem appeared in \cite[Theorem 3.1]{BMSZ} of an early preprint, and it also affected a somewhat similar result in my previous unpublished work \cite{PiyFano3}.
This work is supported by the World Premier International Research Center Initiative (WPI Initiative), MEXT, Japan.


\section{Preliminaries}
\label{sec:prelim}
\subsection{Tilt stability on 3-folds}
\label{subsec:tiltstab}
Let us briefly recall the notions of slope and tilt stabilities  
for a given smooth projective threefold $X$ as
introduced in \cite{BMT}.

Let $H \in \NS(X)$ be an ample divisor class, and  $B \in \NS_{\mathbb{R}}(X)$. The twisted Chern
character with respect to $B$
 is defined by $\ch^B(-) = e^{-B} \ch (-)$.
The twisted slope $\mu_{H , B} $ on  $Coh(X)$ is defined by, for $E \in \Coh(X)$
$$
\mu_{H, B} (E) = \begin{cases}
+ \infty & \text{if } E \text{ is a torsion sheaf} \\
\frac{H^{2} \ch_1^B(E)}{H^3 \ch^B_0(E)} & \text{otherwise}.
\end{cases}
$$
For simplicity we write
$$
\mu_{H} = \mu_{H,0}.
$$
So we have $\mu_{H,B} = \mu_{H} - (BH^2)/(H^3)$.

We say  $E \in \Coh(X)$ is $\mu_{H , B}$-(semi)stable, if for any $0 \ne
F \varsubsetneq E$, $\mu_{H , B}(F)< (\le) \mu_{H ,
  B}(E/F)$. The Harder-Narasimhan property holds for
$\Coh(X)$, and  
 for a given interval $I \subset \mathbb{R} \cup\{+\infty\}$, we  define 
 the subcategory $\HN^{\mu}_{H, B}(I) \subset
\Coh(X)$ by
\begin{equation}
\label{eqn:HN-mu-interval-subcat}
\HN^{\mu}_{H, B}(I) = \langle E \in \Coh(X) : E \text{ is } \mu_{H , B}\text{-semistable with }
\mu_{H , B}(E) \in I \rangle.
\end{equation}
The subcategories  $\tT_{H , B}$ and $\fF_{H , B}$ of $\Coh(X)$ are defined by
\begin{align*}
\tT_{H , B} = \HN^{\mu}_{H, B}((0, +\infty]), \ \ \
\fF_{H , B} = \HN^{\mu}_{H, B}((-\infty, 0]).
\end{align*}
Now $( \tT_{H , B} , \fF_{H, B})$ forms a torsion pair on $\Coh(X)$
and let the abelian category $\bB_{H , B} = \langle \fF_{H , B}[1],
\tT_{H, B} \rangle \subset D^b(X)$ be the corresponding tilt of $\Coh(X)$.

Let $\alpha \in \RR_{>0}$. Following \cite{BMT}, the tilt-slope $\nu_{H, B, \alpha} $ on $\bB_{H , B}$ is defined by, for $E \in \bB_{H,B}$
$$
\nu_{H, B, \alpha}(E) =
\begin{cases}
+\infty & \text{if } H^2 \ch^B_1(E) = 0 \\
\frac{H \ch_2^{B}(E) - (\alpha^2/2) H^3\ch_0(E)}{H^2 \ch^B_1(E)} & \text{otherwise}.
\end{cases}
$$
In \cite{BMT} the notion of $\nu_{H , B, \alpha}$-stability for objects in
$\bB_{H , B}$ is introduced in a similar way to $\mu_{H,
  B}$-stability for $\Coh(X)$. Also it is proved that the abelian
category $\bB_{H , B}$ satisfies the Harder-Narasimhan property with respect to
$\nu_{H , B, \alpha}$-stability. 
Then similar to \eqref{eqn:HN-mu-interval-subcat}  we define the subcategory $\HN^{\nu}_{H, B, \alpha}(I) \subset \bB_{H, B}$ for an
interval $I \subset \mathbb{R} \cup\{+\infty\}$.
The subcategories $\tT_{H , B, \alpha}'$ and $\fF_{H , B, \alpha}'$ of $\bB_{H, B}$ are defined by
$\tT_{H , B, \alpha}' = \HN^{\nu}_{H, B, \alpha}((0, +\infty])$ and $\fF_{H, B}' = \HN^{\nu}_{H, B, \alpha}((-\infty, 0])$. Then the pair $( \tT_{H , B, \alpha}' , \fF_{H , B, \alpha}')$ forms a torsion pair on
$\bB_{H , B}$ and let the abelian category 
\begin{equation}
\label{def:double-tilt}
\aA_{H , B, \alpha} = \langle
\fF_{H , B, \alpha}'[1],\tT_{H , B, \alpha}' \rangle \subset D^b(X)
\end{equation}
 be the corresponding tilt.

\subsection{Some homological algebraic results}

An object of an abelian category is called {minimal} when it has no
proper subobjects or equivalently no nontrivial quotients in the category. 
For example skyscraper sheaves of closed points are the only minimal objects of the abelian category of coherent sheaves on a scheme. 
Moreover, for the abelian category $\bB_{H,B}$ of a 3-fold, we have the following:

\begin{prop}
\label{prop:minimalobj}
The objects  which are isomorphic to the following types are minimal in $\bB_{H,B}$:
\begin{enumerate}
\item skyscraper sheaves $\oO_x$ of $x \in X$.
\item $E[1]$, where $E$ is a $\mu_{H, B}$-stable reflexive sheaf with $\mu_{H, B}(E) =0$.
\end{enumerate}
\end{prop}
\begin{proof}
Similar to the proof of \cite[Proposition 2.2]{Huy}. (Also one can see these objects as examples of the class of minimal objects considered abstractly in \cite[Aside 2.12]{PT}.)
\end{proof}

Let $E, F$ be two objects in the derived category $D^b(X)$ of a smooth projective variety $X$. 
The Euler characteristic $\chi(E,F)$ is defined by 
$$
\chi(E, F) = \sum_{i \in \mathbb{Z}} \hom_{\SX} (E, F[i]).
$$
We write $\chi(\oO_X, E)$ by $\chi(E)$, and so $\chi(E, F) = \chi(E^\vee \otimes F)$. 
The Hirzebruch-Riemann-Roch theorem says, 
\begin{equation}
\label{eqn:RiemannRoch}
\chi(E) = \int_X \ch(E) \cdot \td(X).
\end{equation}
Here $\td(X)$ is the Todd class $\td(T_X)$ of the tangent bundle $T_X$ of $X$. 
When  $X$ is $3$-dimensional, from \cite[Section 4, Appendix A]{Har}
\begin{equation}
\label{eqn:Todd}
\td(X) = 1 +\frac{1}{2} c_1(X) + \frac{1}{12}\left(c_1(X)^2 + c_2(X)\right) +  \frac{1}{24} c_1(X) c_2(X).
\end{equation} 
Here $c_i(X)$ denotes the $i$-th Chern class $c_i(T_X)$ of the tangent bundle $T_X$. 

\subsection{Some sheaf theory}
Let us recall some useful results for coherent sheaves.  
\begin{prop}[{\cite{OSS, HL}}]
\label{prop:dualsheaf}
Let $X$ be an $n$-dimensional smooth projective variety. Then we have the following for $E \in \Coh(X)$:
\begin{enumerate}
\item If  $E \in \Coh_{\le d}(X)$ then it fits into the short exact sequence 
$$
0 \to E_{\le d-1} \to E \to E_d \to 0
$$
in $\Coh(X)$ for some $ E_{\le d-1} \in \Coh_{\le d-1}(X)$ and $ E_d  \in \Coh_d(X)$.
\item $\calExt^i(E,\oO_X) \in \Coh_{\le n-i}(X)$.
\item If $E \in \Coh_d(X)$ then it fits into the short exact sequence 
$$
0 \to E \to \calExt^{n-d}\left( \calExt^{n-d}(E)\right) \to Q  \to 0
$$
in $\Coh(X)$ for some $Q  \in \Coh_{\le d-2}(X)$.
\end{enumerate}
\end{prop}

\begin{lem}[{\cite[Theorem 7.3.1]{HL}, \cite[Theorem 2]{Sim}}]
\label{prop:sheafBG}
Let $X$ be a smooth projective variety of dimension $n \ge 3$ and let $H \in \NS(X)$ be an ample divisor class.
Let $E$ be a $\mu_H$ slope semistable torsion free sheaf on $X$. 
Then we have the following:
\begin{enumerate}
\item  Sheaf $E$ satisfies the so called Bogomolov-Gieseker inequality:
\begin{equation*}
H^{n-2}  \Delta(E) \ge 0, \ \text{where } \Delta(E) = (\ch_1(E))^2 - 2\ch_0(E) \ch_2(E).
\end{equation*}
\item If $H^{n-1} \ch_1(E^{**})=0$ and $H^{n-2} \ch_2(E^{**}) =0$,
then all Jordan-H\"{o}lder slope stable factors of $E^{**}$ are locally free sheaves which have vanishing Chern classes.
\item If $E$ is a $\mu_H$ semistable reflexive sheaf with $H^{n-2}\Delta(E) = 0$, then 
$E$ is a locally free sheaf with $\ch_{i}(E\otimes E^{*}) = 0$ for $i \ge 1$; in particular $\Delta(E) =0$.
\end{enumerate}
\end{lem}

\subsection{Fano 3-folds}
Let us recall some important notions associated to Fano varieties. 
A Fano variety $X$  is a smooth projective variety whose anti-canonical divisor $-K_{X}$ is ample. 
A basic invariant of $X$ is its \textit{index}, this is the maximal integer $r(X)$ such that $K_X$ is divisible by $r(X)$ in $\NS(X)$. So $-K_X = r(X) \cdot H$ for an ample divisor class $H$ in $\NS(X)$. 
The number $d(X) = H^{\dim X}$ is usually called the \textit{degree} of $X$.

If $X$ is an $n$-dimensional Fano variety then  $r(X) \le n +1$. 
Moreover, if $r(X) = n +1$  then
$X \cong \mathbb{P}^n$, and   if $r(X) = n$  then $X$ is a quadric. 
For Fano 3-folds there is an explicit Iskovskikh-Mori-Mukai classification. See \cite[Chapter 12]{Fano} or \cite{MM} for further details.  

Let us collect some basic properties for Fano 3-folds, that we will need in the proceeding sections. 
\begin{prop}
\label{prop:Fano3-prop}
Let $X$ be a Fano 3-fold of index $r(X) = r$ and degree $d(X) = d$. Then we have the following:
\begin{enumerate}
\item $h^i(\oO_X) =0$ for all $i>0$, and $\chi(\oO_X) = 1$. 
\item $H \cdot c_2(X) = 24/r$. 
\item $\td(X) = \left(1, \, \frac{1}{2}r H, \, \frac{1}{12}\left( r^2 H^2 + c_2(X)  \right) , \, 1\right)$. 
\end{enumerate}
\end{prop}
\begin{proof}
Since $-K_X$ is ample, from the Kodaira's vanishing theorem $H^i(\oO_X, X) =0$ for all $i >0$. So 
we have $\chi(\oO_X) = h^0(\oO_X) =1$. 

Let us compute the Todd class of the tangent bundle $T_X$ of $X$. Since the cotangent bundle is $\Omega_X \cong  T_X^{*}$, $c_1(X) = -c_1(\Omega_X)$. Also $\omega_X = \det(\Omega_X)$ and so 
$c_1(X) = -c_1(\omega_X) = -K_X = rH$. 
From the Hirzebruch-Riemann-Roch theorem \eqref{eqn:RiemannRoch}, $\chi(\oO_X) = \int_X \ch(\oO_X) \cdot \td(X)$, and so 
$ \frac{1}{24}c_1(X) c_2(X) =1$. 
The required expression for Todd class follows from \eqref{eqn:Todd}. 
\end{proof}

\section{Some Properties of Tilt Stable Objects}
\label{sec:tiltproperties}
\subsection{Some slope bounds for tilt stable objects}
\label{sec:tiltslopebounds}
Let $X$ be a smooth projective 3-fold.
We follow the same notations for tilt stability introduced in Section \ref{subsec:tiltstab} for $X$. 

By construction, $\Coh_{\le 2}(X) \subset \bB_{H, B}$. 
Moreover, we have the following for its subcategory $\Coh_{\le1}(X)$. 
\begin{prop}
\label{prop:Tor1inftyslope}
We have  
$\Coh_{\le1}(X) \subset \HN^{\nu}_{H, B, \alpha} (+\infty)$.
\end{prop}
\begin{proof}
Let $E \in \Coh_{\le1}(X)$. Assume the opposite for a contradiction; so that $0 \to P \to E\to Q \to 0$ is a short exact sequence on $\bB_{H,B}$ with $\nu_{H, B, \alpha} (Q) < +\infty$.
 By considering the long exact sequence of $\Coh(X)$ cohomologies we have 
$\hH^{-1}(P) = 0$, and since $\ch_1(E) = 0$, 
$\ch_1 (\hH^{-1}(Q)) = \ch_1(\hH^{0}(P))$. Since $\hH^{-1}(Q) \in \HN^{\mu}_{H, B} ((-\infty, 0])$ and 
$\hH^{0}(P) \in \HN^{\mu}_{H, B} ((0, +\infty])$, we have $H^2 \ch_1^B(\hH^0(P))=0$. 
So  $\hH^{0}(P)  \in \Coh_{\le1}(X) $, and $\hH^{-1}(Q) =0$.
Hence, $Q \cong \hH^0(Q)$ is a quotient sheaf of $E \in \Coh_{\le1}(X)$; that is 
$Q \in \Coh_{\le1}(X)$.
Therefore, $\nu_{H, B, \alpha}(Q) = +\infty$. This is the required contradiction. 
\end{proof}

\begin{prop}
\label{prop:reflexivityatminus1place}
Let $E \in \HN^{\nu}_{H, B, \alpha}((-\infty,+\infty))$. Then $\hH^{-1}(E)$ is a reflexive sheaf.
\end{prop}
\begin{proof}
For $E \in \HN^{\nu}_{H, B, \alpha}((-\infty,+\infty))$, let us denote $E_j = \hH^j(E)$. 
Object $E$ fits into the short exact sequence
$ 0\to E_{-1} [1] \to E \to E_0 \to 0$
 in $\bB_{H,B}$. 
Here $E_{-1}$ is torsion free and so it fits into the short exact sequence $0 \to E_{-1} \to E_{-1}^{**} \to Q \to 0$ in $\Coh(X)$
 for some $Q \in \Coh_{\le 1}(X)$. Therefore, $0 \to Q \to E_{-1}[1] \to E_{-1}^{**}[1] \to 0$ is a short exact sequence in $\bB_{H,B}$. Hence $Q$ is a subobject of $E\in \HN^{\nu}_{H, B,\alpha}((-\infty,+\infty))$. By Proposition \ref{prop:Tor1inftyslope}, $Q \in \HN^{\nu}_{H, B,\alpha}(+\infty)$, and so $Q =0$. That is $E_{-1}$ is reflexive. 
 \end{proof}

\begin{defi}
\label{def:discriminant}
For $E \in D^b(X)$ we define
\begin{align*}
\Delta(E) & = (\ch_1(E))^2- 2 \ch_0(E) \ch_2(E) \in H^4(X, \ZZ) ,\\
\overline{\Delta}_{H, B}(E) & = (H^2 \ch_1^B(E))^2 - 2 H^3 \ch_0(E)  H \ch_2^B(E).
\end{align*}
We simply write $\overline{\Delta}_{H}= \overline{\Delta}_{H, 0}$.
\end{defi}
We have $\overline{\Delta}_{H, B} = H^3 H \cdot \Delta(E) + (H^2 \ch_1^B(E))^2 - H^3 H (\ch_1^B(E))^2$. 
From the Hodge index theorem, $(H^2 \ch_1^B(E))^2 - H^3 H (\ch_1^B(E))^2 \ge 0$, 
and so
\begin{equation*}
\overline{\Delta}_{H, B} (E) \ge  H^3 H \cdot \Delta(E).
\end{equation*}

For any $t \in \RR$ we have 
\begin{align}
\label{eqn:nu-to-0-stab-para}
\nu_{H, B, \alpha}- t & = \frac{H \ch_2^B- (1/2) \alpha^2H^3 \ch_0}{H^2 \ch_1^B} -t \\
& =  \frac{H \ch_2^{B+tH}  - (1/2)(t^2+\alpha^2)H^3 \ch_0 }{H^2 \ch_1^B }. \nonumber
\end{align}

Let us recall the following slope bounds from \cite{PT} for cohomology sheaves of complexes in the abelian category  $\bB_{H,B}$. 
\begin{prop}[{\cite[Proposition 3.13]{PT}}]
\label{prop:slope-bounds} 
Let $E \in \bB_{H, B}$ and $E_{i}= \hH^{i}(E)$.
Then we have the following:
 \begin{enumerate}
\item if $E \in \HN^{\nu}_{H, B, \alpha}((-\infty, t)) $,  then 
$E_{-1} \in \HN^{\mu}_{H, B} ((-\infty, t- \sqrt{t^2 + \alpha^2}))$;

\item if $E \in \HN^{\nu}_{H, B}((t, +\infty)) $,  then 
$E_0 \in \HN^{\mu}_{H, B} ((t + \sqrt{t^2 + \alpha^2}, +\infty])$;
and

\item if $E$ is tilt semistable with $\nu_{H,B, \alpha}(E) =t$, then
\begin{enumerate}
\item $E_{-1} \in \HN^{\mu}_{H, B} ((-\infty, t- \sqrt{t^2 + \alpha^2}])$ with equality
$\mu_{H,B}(E_{-1}) = t- \sqrt{t^2 + \alpha^2}$ holds
  if and only if $H^2 \ch_2^{B+(t- \sqrt{t^2 + \alpha^2})H}(E_{-1}) = 0$, that is 
  when $\Del_{H,B}(E_{-1})=0$,
   and
\item when $E_0$ is torsion free $E_0 \in \HN^{\mu}_{H, B} ([t+ \sqrt{t^2 + \alpha^2}, +\infty))$ with equality
$\mu_{H,B}(E_{0}) =t+ \sqrt{t^2 + \alpha^2}$ holds
  if and only if $H^2 \ch_2^{B+(t- \sqrt{t^2 + \alpha^2})H}(E_{0}) =0$, that is 
  when $\Del_{H,B}(E_0)=0$.
\end{enumerate}
\end{enumerate}
\end{prop}

\begin{defi}
\label{def:Psi}
For an object $E$ and $\delta \in \RR_{\ge 0}$, we define
\begin{align*}
\Psi_{H,B,\alpha, \delta}^{\pm}(E) = \nu_{H,B,\alpha}(E) \pm \sqrt{ (\nu_{H,B,\alpha}(E))^2 + \alpha^2 + \delta}.
\end{align*}
\end{defi}

\begin{prop}
\label{prop:muslopeboundfortiltstability}
Let $E \in \bB_{H,B}$ be a tilt stable object with $\nu_{H,B,\alpha}(E) = t < + \infty$. 
Then we have the following:
\begin{enumerate}
\item $H^2 \ch_1^{B + tH}(E) - \sqrt{t^2 + \alpha^2} H^3 \ch_0(E) \ge 0$, with equality holds when $\hH^{-1}(E) =0$ and $\hH^0(E)$ is a slope stable torsion free sheaf such that 
$\hH^0(E)^{**}$ is locally free with $\Del_{H,B} =0$. In particular, when $\ch_0(E) >0$, 
$$
\mu_{H,B}(E) \ge t +  \sqrt{t^2 + \alpha^2} = \Psi_{H,B,\alpha, 0}^{+}(E).
$$
\item $H^2 \ch_1^{B + tH}(E) + \sqrt{t^2 + \alpha^2} H^3 \ch_0(E) \ge 0$, 
with equality holds when $\hH^{0}(E) =0$ and $\hH^{-1}(E)$ is a slope stable locally  free sheaf with $\Del_{H,B} =0$. In particular, when $\ch_0(E) < 0$, 
$$
\mu_{H,B}(E) \le t -  \sqrt{t^2 + \alpha^2} = \Psi_{H,B,\alpha, 0}^{-}(E) .
$$
\item $\Del_{H,B}(E) \ge 0$; if $\Del_{H,B}(E) =0$ then 
$E$ is isomorphic to either $E_{-1}[1]$ for some $\mu_H$  stable locally free sheaf $E_{-1}$, or $\mu_H$ stable torsion free sheaf $E_0$ such that $E_0^{**}$ is locally free with $E_{0}^{**}/E_{0} \in \Coh_0(X)$.
\end{enumerate}
\end{prop}

\begin{proof}
We have $H \ch_2^{B + tH}(E) = (1/2)(t^2 + \alpha^2)H^3\ch_0(E)$.
Let us denote $E_{i} = \hH^i(E)$ and 
$$
D_i = H^2 \ch_1^{B+tH}(E_i), \ \text{  } \ R_i = \sqrt{t^2 + \alpha^2} H^3 \ch_0(E_i).
$$
From Proposition \ref{prop:slope-bounds}, $D_0 \ge R_0$ and $D_{-1} \le - R_{-1}$. 

\noindent (i) \ We have 
\begin{align*}
H^2 \ch_1^{B + tH}(E) - \sqrt{t^2 + \alpha^2} H^3 \ch_0(E) & =  (D_0 - D_{-1}) - (R_0 - R_{-1}) \\ 
& = (D_ 0 - R_0) + (- D_{-1} + R_{-1}) \\
& \ge (D_ 0 - R_0) + 2 R_{-1} \ge 0.
\end{align*} 
Here the equality in the last ``$\ge$''  holds when $D_0 =R_0$ and $R_{-1} =0$. Let us consider this case.
We have $E_{-1} =0$, that is $E \cong E_0$. 
Let us prove $E_0$ is a slope stable torsion free sheaf. 
Assume the opposite; so there exists a slope stable quotient sheaf $G$ of $E_0$   with
$\mu_{H,B}(G) \le t + \sqrt{t^2 + \alpha^2}$. 
Moreover, $E _0 \twoheadrightarrow G$ is also a surjection in
$\bB_{H,B}$, and since $E_0$ is tilt stable, $G \in  \HN^{\nu}_{H,B,\alpha}((t, +\infty])$. From (ii) of Proposition \ref{prop:slope-bounds}, $\mu_{H,B}(G) > t + \sqrt{t^2 + \alpha^2}$; this is not possible. Hence $E_0$ is slope stable. 

So $\overline{\Delta}_{H, B}(E_{0}) = 0$. 
From Lemma \ref{prop:sheafBG}, slope stable reflexive sheaf $E_{0}^{**}$ is  locally free with $\overline{\Delta}_{H, B}=0$; hence, $E_0^{**}/E_0 \in \Coh_0(X)$.  \\ 

\noindent (ii) Proof is similar to that of (i).  \\

\noindent (iii) \ We have 
\begin{align*}
\overline{\Delta}_{H, B}(E) & = (H^2 \ch_1^{B+tH}(E))^2- (t^2 + \alpha^2) (H^3 \ch_0(E))^2   \\
& = \left\{H^2 \ch_1^{B + tH}(E) - \sqrt{t^2 + \alpha^2} H^3 \ch_0(E) \right\} \times
       \left\{H^2 \ch_1^{B + tH}(E) + \sqrt{t^2 + \alpha^2} H^3 \ch_0(E) \right\}. 
\end{align*}
From   (i) and (ii), $\overline{\Delta}_{H, B}(E) \ge 0$; and  the equality holds when we have the equalities in either  (i) or (ii). 
\end{proof}

\begin{prop}
\label{prop:PsiDeltarelation}
Let $E$ be an $\nu_{H,B,\alpha}$ semistable object in $ \bB_{H,B}$ with  $\nu_{H,B,\alpha}(E) < +\infty$, 
$\ch_0(E) >0$,  and let $\lambda_1, \lambda_2$ be some non-negative  constants. We have
\begin{equation*}
\label{eqn:Del-frac-ineq}
 \lambda_1 <   \frac{\Del_{H,B}}{(H^3 \ch_0(E))^2} <   \lambda_2
\end{equation*}
if and only if 
\begin{equation*}
\label{eqn:mu-Psi-ineq}
 \Psi_{H,B,\alpha, \lambda_1}^{+}(E) <  \mu_{H,B}(E)<  \Psi_{H,B,\alpha, \lambda_2}^{+}(E). 
\end{equation*}
Moreover, if we have one of the above equivalent inequalities for $E$, then
$$
\frac{\partial \Psi_{H,B,\alpha, \lambda_1}^{+}(E)}{\partial \alpha} > 0 > \frac{\partial \Psi_{H,B,\alpha, \lambda_2}^{+}(E)}{\partial \alpha}.
$$
\end{prop}
\begin{proof}
As in the proof of (iii) of Proposition \ref{prop:muslopeboundfortiltstability}, we have 
$$
\Del_{H, B} (E) = \left( H^2 \ch_{1}^{B + \nu_{H,B, \alpha}(E) H}(E) \right)^2 - \left( (\nu_{H,B, \alpha}(E))^2 + \alpha^2 \right) (H^3\ch_0(E))^2.
$$
From Proposition \ref{prop:muslopeboundfortiltstability}, $H^2 \ch_{1}^{B + \nu_{H,B, \alpha}(E) H}(E) \ge 0$ and  since $\ch_0(E) >0$,  by direct computation one can get the required inequalities in both directions. 

By differentiating $\nu_{H,B,\alpha}(E)$ with respect to $\alpha$ we get
$$
\frac{\partial \nu_{H,B,\alpha}(E)}{\partial \alpha} = \frac{-\alpha}{\mu_{H,B}(E)}.
$$
By differentiating  $\Psi_{H,B,\alpha, \lambda_1}^{+}(E)$ with respect to $\alpha$ we get
\begin{align*}
\frac{\partial \Psi_{H,B,\alpha, \lambda_1}^{+}(E)}{\partial \alpha} & = 
\frac{\partial \nu_{H,B,\alpha}(E)}{\partial \alpha} + \frac{1}{ \sqrt{(\nu_{H,B, \alpha}(E))^2 + \alpha^2 + \lambda_1}} \left( \alpha + \nu_{H,B,\alpha}(E) \cdot \frac{\partial \nu_{H,B,\alpha}(E)}{\partial \alpha} \right) \\
& = \frac{\alpha}{\mu_{H,B}(E) \sqrt{(\nu_{H,B, \alpha}(E))^2 + \alpha^2 + \lambda_1}} \left(\mu_{H,B}(E) - \Psi_{H,B,\alpha, \lambda_1}^{+}(E)   \right) >0 
\end{align*}
as required. Similarly one can get   the other inequality. 
\end{proof}

\begin{rmk}
\rm
One can have a similar Proposition considering $\ch_0(E) <0$ case involving $\Psi_{H,B,\alpha, \lambda_i}^{-}(E)$. 
\end{rmk}

Recall the following result about the walls for tilt stable objects from \cite{PT}:
\begin{prop}[{\cite[Lemma 3.15]{PT}}]
\label{prop:wall-tiltstable}
Let $E \in \bB_{H,B}$ be a tilt stable object with $\nu_{H,B,\alpha}(E)  < +\infty$. 
Then $E \in \bB_{H, B+ b H}$ is 
$\nu_{H, B +b H, a}$-stable for all $a \in \mathbb{R}_{>0}$ and $b \in \mathbb{R}$ such that
$$
a^2  +  \left( b -\nu_{H,B,\alpha}(E) \right)^2  =(\nu_{H,B,\alpha}(E))^2 + \alpha^2.
$$
\end{prop}

The following results are crucial for us.
\begin{prop}[{\cite[Lemma 2.7]{BMS}}] 
\label{prop:limittiltstableobjects}
Let $E \in \bB_{H,B}$ be  $\nu_{H,B,\alpha}$ tilt stable for all $\alpha \ge \alpha_0$ for some 
$\alpha_0>0$ with 
$\nu_{H,B,\alpha_0}(E) < +\infty$. Then we have the following:
\begin{enumerate}
\item If $\ch_0(E)> 0$ then $E$ is a slope semistable torsion free sheaf.
\item If $\ch_0(E) =0$ then $E$ is a  slope semistable pure torsion sheaf in $\Coh_{\le 2}(X)$. 
\item If $\ch_0(E) <0$ then $\hH^{-1}(E)$ is a slope semistable reflexive sheaf and 
$\hH^0(E) \in \Coh_{\le 1}(X)$. 
\end{enumerate}
\end{prop}

\begin{prop}[{\cite[Proposition 7.4.1]{BMT}}] 
\label{prop:Delta0mustabletiltstable}
If $E$ is a $\mu_{H, B}$-(semi)stable sheaf with $\overline{\Delta}_{H, B}(E) =0$, then
$E$ or $E^{**}[1]$ in $\bB_{H, B}$ is  $\nu_{H,B, \alpha}$-(semi)stable.
\end{prop}

We have following result for certain short exact sequences in $\bB_{H,B}$.
\begin{prop}
\label{prop:discrimi-ses-decrease}
Let $0 \to E_1 \to E \to E_2 \to 0$ be a short exact sequence in $\bB_{H,B}$ such that 
$E, E_1, E_2$ are $\nu_{H,B, \alpha}$-semistable
with $\nu_{H,B,\alpha}(E_1) = \nu_{H,B, \alpha}(E_2) < +\infty$. Then 
$$
\overline{\Delta}_{H, B}(E) \ge \overline{\Delta}_{H, B}(E_1) + \overline{\Delta}_{H, B}(E_2),
$$
where the equality holds when $\overline{\Delta}_{H, B}= 0$ for $E, E_1, E_2$. 

Moreover, when $B \in \RR\langle H \rangle$,  if 
$\Del_{H,B}(E_i) =0$ and $\Del_{H,B}(E_j) >0$ for $i \neq j$ with $i, j \in \{1,2\}$, then 
$$
\overline{\Delta}_{H, B}(E) \ge \overline{\Delta}_{H, B}(E_j)+  1.
$$
\end{prop}
\begin{proof}
If $\nu_{H,B,\alpha}(E_1) = \nu_{H,B, \alpha}(E_2) = \vt$ for some $\vt < +\infty$ then from Proposition \ref{prop:wall-tiltstable},
$E_1, E_2 \in \bB_{H, B+ \vt H}$
are $\nu_{H,B+ \vt H, \sqrt{\vt^2 + \alpha^2}}$ tilt semistable with 
zero tilt slopes. Therefore, we can assume $\nu_{H,B, \alpha}(E_1) = \nu_{H,B,\alpha}(E_2)=0$. 

 Let us write, for $i=1,2$:
$$
A_i = H^2 \ch_1^B(E_i), \ \text{ and } \ B_i = \alpha H^3 \ch_0(E_i).
$$
Since $E_1, E_2$ are tilt semistable with zero tilt slopes, from Proposition \ref{prop:muslopeboundfortiltstability} we have 
$A_ i + B_i \ge 0$ and $A_i - B_i \ge 0$ for $i=1,2$. 
Therefore, 
\begin{equation*}
2(A_1 A_2 - B_1 B_2)  = (A_1 -B_1)(A_2 + B_2) + (A_1 + B_1)(A_2 -B_2)  \ge 0.
\end{equation*}
Since $E, E_1, E_2$ have zero tilt slopes,
\begin{align*}
\overline{\Delta}_{H, B}(E)  & = (A_1 + A_2)^2 - (B_1 + B_2)^2  = \sum_{i=1}^2 (A_1^2- B_i^2) +  2(A_1 A_2 - B_1 B_2)   \ge  \sum_{i=1}^2 \overline{\Delta}_{H, B}(E_i);
\end{align*}
and the equality holds when $A_i= B_i$ or $A_i =-B_i$ for each $i=1,2$; hence, $\overline{\Delta}_{H, B}= 0$ for $E, E_1, E_2$.

Suppose $B \in \RR \langle H \rangle$. Let us  consider the case $\Del_{H,B}(E_1)=0$ and 
$\Del_{H,B}(E_2)>0$, and the arguments for the other case is similar. 
Since $\Del_{H,B}(E_1)=0$ and $A_1 >0$, we have either $A_1 = B_1$ or $A_1 = -B_1$. 
Also $\Del_{H,B} = \Del_{H,0}=\Del_H$ and it is integral. Therefore,
$$
\Del_{H,B}(E) - \Del_{H,B}(E_2) = (A_1 -B_1)(A_2 + B_2) + (A_1 + B_1)(A_2 -B_2) \ge 1.
$$
This completes the proof. 
\end{proof}

\begin{lem}
\label{prop:Delta0tiltstableFano3}
Suppose $X$ is a Fano 3-fold of index $r$; so $-K_X = rH$ for some ample divisor class $H$. 
Any $\mu_H$ stable reflexive sheaf with $\Del_{H} = 0$ are line bundles $\oO_X(mH)$ for some $m \in \ZZ$.

Therefore, from (iii) of Proposition  \ref{prop:muslopeboundfortiltstability}, if $E$ is a $\nu_{H, \beta H, \alpha}$ tilt  stable object on $X$ with $\Del_{H}(E) = 0$,  then $E$ is isomorphic to $\oO_X(mH)[1]$ or $\iI_Z(mH)$ for some $m \in \ZZ$ and $0$-subscheme $Z \subset X$.
\end{lem}
\begin{proof}
Let $E$ be a $\mu_H$ stable reflexive sheaf with $\Del_{H} = 0$. 
Let $k$ be the rational number defined by
$$
k =\frac{  H^2 \ch_1(E) }{H^3 \ch_0(E)}.
$$
From Proposition \ref{prop:Delta0mustabletiltstable} $E$ and $E[1]$ are tilt stable. In particular, let us  consider the tilt slopes with respect to the stability parameters 
\begin{align*}
\beta & = k - (r/2), \\
\alpha & = r/2 -\vep
\end{align*}
for some $\vep \in (0, \ r/2)$.
We have $E, E(-rH)[1] \in \bB_{H,\beta H}$, and by direct computation,
$$
\nu_{H,\beta H, \alpha}(E) > 0 > \nu_{H,\beta H, \alpha}(E(-rH)[1])
$$
So we have 
$$
\Hom_{\SX}(E, E(-rH)[1]) = 0.
$$
Since $\omega_X = \oO_X(-rH)$, from the Serre duality, $\Hom_{\SX}(E, E[2]) = 0$. 
Since $E$ is tilt stable $\Hom_{\SX}(E, E) \cong \CC$; therefore
\begin{align*}
\chi(E, E) & = \sum_{i} \hom_{\SX}(E, E[i]) \le \hom_{\SX}(E, E) +  \hom_{\SX}(E, E[2]) = \hom_{\SX}(E, E)  = 1.
\end{align*}
On the other hand from the Riemann-Roch formula \eqref{eqn:RiemannRoch}
\begin{align*}
\chi(E, E) = \int_X \ch(E) \ch(E^{\vee}) \td_X =-\frac{r}{2} H (\ch_1(E)^2 - 2\ch_0(E) \ch_2(E)) + \ch_0(E)^2.
\end{align*}
Therefore, $\chi(E, E)  \le 1$ implies 
\begin{align*}
\ch_0(E)^2 \le 1 + \frac{r}{2}  H \Del(E).
\end{align*}
Since $0 = \Del_H (E) \ge H^3 H \Del(E) \ge 0$,  we have $H \Del(E)  =0$.
Therefore, $\ch_0(E) \le 1$. Since $E$ is torsion free $\ch_0(E)$ is integral, and so we have $\ch_0(E) =1$. Also the reflexivity of $E$ implies it is a line bundle. 
So $ \ch(E) = e^{D}$ for some $D \in \NS(X)$. 
Since $ \Del_H (E) =0$, we have $(H^2D)^2 = H^3 HD^2$, that is $D \in \ZZ\langle H \rangle$ as required. 
This completes the proof.
\end{proof}

\subsection{Tilt stability under dualizing}
\label{sec:dualtiltstability}
\begin{nota}
\rm
For $E \in \bB_{H,B}$ we write
$$
E^i = H^{i}_{\bB_{H,-B}}(E^\vee). 
$$
So for example $E^{12} =H^{2}_{\bB_{H, B}} \left( \left(H^{1}_{\bB_{H,-B}}(E^\vee)\right)^\vee\right)$.
\end{nota}
We have the following. 
\begin{prop}
\label{prop:dual-B-object}
Let $E \in \HN^{\nu}_{H, B, \alpha}((-\infty,+\infty))$. Then $E^i = 0$ for $i \ne 1,2$ with $E^2 \in \Coh_0(X)$.
\end{prop}
\begin{proof}
For $E \in \HN^{\nu}_{H, B, \alpha}((-\infty,+\infty))$, let us denote $E_j = \hH^j(E)$. 
Object $E$ fits into the short exact sequence
\begin{equation*}
\label{ses:tstruct-B-object}
 0\to E_{-1} [1] \to E \to E_0 \to 0
\end{equation*}
 in $\bB_{H,B}$. 
From Proposition \ref{prop:reflexivityatminus1place},  $E_{-1}$ is a reflexive sheaf. 

By dualizing the above short exact sequence, we have the following distinguished triangle
\begin{equation}
\label{ses:dual-triangle}
E_0^{\vee} \to E^\vee \to E_{-1}^{\vee}[-1] \to E_0^{\vee}[1].
\end{equation}

By considering $t \to +\infty$ in (i) of Proposition \ref{prop:slope-bounds},  we have $E_{-1} \in \HN^{\mu}_{H,B}((-\infty,0))$.  
So $E_{-1}^* \in  \HN^{\mu}_{H,B}((0, +\infty))$. Since $E_{-1}$ is reflexive
$\calExt^1(E_{-1}, \oO_X) = \hH^{1}(E_{-1}^{\vee}) \in \Coh_0(X)$ and 
$\calExt^i(E_{-1}, \oO_X) = \hH^{i}(E_{-1}^{\vee}) = 0$ for $i \ge 2$.
Therefore, $\left(E_{-1}[1]\right)^i =0$ for $i \ne 1,2$. 

The sheaf $E_0 \in \HN^{\mu}_{H,B}((0, +\infty])$ and so $E_0^* \in \HN^{\mu}_{H,-B}((-\infty,0)) \subset \bB_{H,-B}[-1]$. Moreover, for $i\ge 1$,  $\calExt^i(E_{0}, \oO_X) = \hH^{i}(E_{0}^{\vee}) \in \Coh_{\le (3-i)}(X) \in \bB_{H,-B}$. So $E_0^i = 0$ for $i \ne 1,2,3$ with $E_0^2 \in \Coh_{\le 1}(X)$ and $E_0^3 \in \Coh_0(X)$.
Therefore, by considering the long exact sequence of $\bB_{H,-B}$-cohomologies associated to the triangle \eqref{ses:dual-triangle}, we have 
$E^{i} = 0$ for  $i \ne 1,2,3$ with $E^2 \in \Coh_{\le 1}(X)$ and $E^3 \in \Coh_0(X)$. 

For any $x \in X$, 
\begin{align*}
\Hom_{\SX}(E^3, \oO_x) & \cong \Hom_{\SX}(E^\vee[3], \oO_x ) \\
 & \cong \Hom_{\SX}(E^\vee, \oO_x[-3] ) \\
 & \cong \Hom_{\SX}(\left(\oO_x[-3]\right)^\vee , E  ) \\
 & \cong \Hom_{\SX}(\oO_x, E) = 0,
\end{align*}
as the skyscraper sheaf $\oO_x \in \Coh_0(X) \subset \HN^{\nu}_{H,B, \alpha}(+\infty)$ and $E \in \HN^{\nu}_{H, B, \alpha}((-\infty,+\infty))$. Therefore, $E^3  =0$. 

For any $T \in Coh_1(X)$, 
\begin{align*}
\Hom_{\SX}(E^2, T) & \cong \Hom_{\SX}(E^\vee[2], T ) \\
 & \cong \Hom_{\SX}(E^\vee, T[-2] ) \\
 & \cong \Hom_{\SX}(\left(T[-2]\right)^\vee , E  ) \\
 & \cong \Hom_{\SX}(\calExt^2(T, \oO_X), E) = 0,
\end{align*}
as $\calExt^2(T, \oO_X) \in \Coh_1(X) \subset \HN^{\nu}_{H,B, \alpha}(+\infty)$ and $E \in \HN^{\nu}_{H, B, \alpha}((-\infty,+\infty))$. Therefore, $E^2 \in \Coh_0(X)$. This completes the proof. 
\end{proof}

\begin{prop}
\label{prop:B-obj-like-sheaves-dual}
We have the following for $E \in \HN^{\nu}_{H,B,\alpha}((-\infty,+\infty))$:
\begin{enumerate}
\item $E$ fits into the short exact sequence 
$$
0 \to E \to E^{11} \to E^{23} \to 0
$$
in $\bB_{H,B}$, where $ E^{23} \in \Coh_0(X)$, 
\item $E^{1,k}=0$ for $k \ne 1$, 
\item $\Hom_{\SX}(\Coh_{\le 1}(X), E^1) = 0$, and
\item $\Hom_{\SX}(\Coh_{0}(X), E^1[1]) = 0$.
\end{enumerate}
\end{prop}
\begin{proof}
By Proposition \ref{prop:dual-B-object}, $E^i = 0$ for $i \ne 1,2$ and $E^2 \in \Coh_0(X)$. So $(E^2)^{\vee} \cong E^{23}[-3]$.

Since $E^{\vee \vee} \cong E$, we have the spectral sequence:
\begin{equation*}
\label{specseq:doubledual-B-obj}
H^{p}_{\bB_{H, B}}\left(\left( H^{-q}_{\bB_{H, -B}} (E^{\vee})\right)^{\vee}\right) \Longrightarrow H^{p+q}_{\bB_{H, B}}(E).
\end{equation*}
Consider the convergence of this spectral sequence for $E \in \HN^{\nu}_{H,B, \alpha}((-\infty,+\infty))$.
From the convergence we get  $E^{1,k} = 0$ for $k \ne 1$ and also we have the short
 exact sequence
$
0 \to E \to E^{11} \to E^{23} \to 0
$
in $\bB_{H,B}$. 

For $T \in \Coh_{\le 1}(X)$, $\calExt^i(T, \oO_X) \in \Coh_{(3-i)}(X)$; and so 
$T^{\vee}\in \langle \bB_{H,B}[-2], \bB_{H,B}[-3]\rangle$. 
On the other hand $(E^{1})^{\vee} \in \bB_{H,B}[-1]$. 
Hence, 
$$
\Hom_{\SX}(T, E^1) \cong \Hom_{\SX}((E^1)^\vee, T^\vee)=0
$$
as required in part (iii). 

For any skyscraper sheaf $\oO_x$ of $x \in X$, we have 
$$
\Hom_{\SX}(\oO_x, E^1[1]) \cong \Hom_{\SX}((E^1[1])^\vee, \oO_x^\vee) \cong \Hom_{\SX}(E^{11}[-2] , \oO_x[-3]) = 0
$$
as required in part (iv). 
\end{proof}

\begin{prop}
\label{prop:stability-B-dual}
Let $E  \in \HN^{\nu}_{H,B,\alpha}((-\infty,+\infty))$. Then 
\begin{enumerate}
\item $E$ is $\nu_{H,B,\alpha}$-stable (resp. $\nu_{H,B,\alpha}$-semistable)  if and only if $E^{11}$  is $\nu_{H,B,\alpha}$-stable  (resp. $\nu_{H,B,\alpha}$-semistable),
\item  $\nu_{H, -B,\alpha}(E^1) = -\nu_{H, B,\alpha}(E) $,
\item $E$ is $\nu_{H,B,\alpha}$-stable  (resp. $\nu_{H,B,\alpha}$-semistable)  if and only if $E^1$ is 
$\nu_{H, -B,\alpha}$-stable  (resp. $\nu_{H,B,\alpha}$-semistable), and
\item  $E^1 \in \HN^{\nu}_{H,-B,\alpha}((-\infty,+\infty))$. 
\end{enumerate}
\end{prop}
\begin{proof}
From part (2) of Proposition 3.5 in \cite{LM}, we have (i). 

By Proposition \ref{prop:dual-B-object} and from definition of the twisted Chern character we have   
\begin{align*}
- \ch^{-B}(E^1)  + \ch^{-B}(E^2) &  = \ch^{-B}(E^\vee)= e^{B}\ch(E^{\vee}) =  (e^{-B}\ch(E) )^{\vee} \\ 
    & = (ch^B(E))^\vee = (\ch_0^B(E), -\ch_1^B(E),\ch_2^B(E),- \ch_3^B(E)).
\end{align*}
Since  $E^2 \in \Coh_0(X)$, we have $\nu_{H, -B,\alpha}(E^1) = -\nu_{H, B,\alpha}(E) $.

Let $E \in \bB_{H,B}$ be a $\nu_{H,B,\alpha}$-semistable object. Assume  $E^{1} \in \bB_{H,-B}$ is
 $\nu_{H,-B,\alpha}$-unstable. From the Harder-Narasimhan filtration there exists a quotient $E^1 \twoheadrightarrow Q$ in $\bB_{H,-B}$, where $Q$ is  the lowest $\nu_{H,-B,\alpha}$-semistable Harder-Narasimhan factor. Since 
 $\nu_{H, -B, \alpha}(E^1) = -\nu_{H, B, \alpha}(E) $, $\nu_{H, -B, \alpha}(Q) <\nu_{H, -B, \alpha}(E^1) < +\infty$. 
 By (ii), $\nu_{H, B, \alpha}(Q^1) > \nu_{H, B, \alpha}(E^{11}) $ with $Q^1 \hookrightarrow E^{11}$ in $\bB_{H,B}$; this is not possible as $E^{11}$ is $\nu_{H,B, \alpha}$-semistable by (i). 
 
 Part (iv) is a direct consequence of (iii). 
\end{proof}
Consequently, we have the following:
\begin{cor}
\label{prop:reduce-BGineq-class}
We only need to check Bogomolov-Gieseker type inequalities in \cite{BMT, BMS, PT},  Conjectures \ref{conj:BG-ineq-single} and \ref{conj:BG-ineq}   for tilt stable objects $E$ satisfying
\begin{itemize}
\item $E \cong E^{11}$
\item $\ch_0(E) \ge 0$ (or $\ch_0(E)\le 0$).
\end{itemize}
\end{cor}
\begin{aside}
\rm
Let $E$ be an $\nu_{H,B,\alpha}$-stable object in $\bB_{H,B}$ with $\nu_{H,B, \alpha}(E) = 0$. By Proposition \ref{prop:B-obj-like-sheaves-dual}, it fits into the short exact sequence 
$0 \to E \to E^{11} \to E^{23} \to 0$ in $\bB_{H,B}$ with $E^{23} \in \Coh_0(X)$. 
Moreover, by Proposition \ref{prop:stability-B-dual}, $E^{11} \in  \bB_{H,B}$ is $\nu_{H,B,\alpha}$-stable with 
$\nu_{H,B, \alpha}(E^{11}) = 0$. Also by Proposition \ref{prop:B-obj-like-sheaves-dual}, $\Hom_{\SX}(\Coh_0(X), E^{11}[1]) = 0$.
Hence by \cite[Lemma 2.3]{MP1} or \cite[Aside 2.12]{PT}, $E^{11}[1] \in \aA_{H,B, \alpha}$ is a minimal object. 
\end{aside}
\section{Bogomolov-Gieseker  Type Inequality Conjecture for 3-folds}
\label{sec:BGconj}
In this section we let $X$ be a  smooth projective 3-fold.

\subsection{Modified conjectural inequality}
Let us  modify the   Bogomolov-Gieseker type inequality conjecture  for our smooth projective 3-fold $X$.

First we introduce the expression of the inequality as follows:
\begin{defi}
\label{def:BGineq-term}
Let us fix classes $H \in \NS(X)$, $B \in \NS_{\RR}(X)$ such that 
$H$ is ample, and
 $\Lambda \in H^4(X, \QQ)$ satisfying  $\Lambda \cdot H =0$.  
For $\xi \in \mathbb{R}_{\ge 0}$, $ \alpha \in \mathbb{R}_{>0}$ and $\beta \in \mathbb{R}$, we define
$$
D^{B, \xi}_{\alpha, \beta}(E) = \ch_3^{B + \beta H}(E) + \left(  \Lambda - \left( \xi+  \frac{1}{6}\alpha^2 \right) H^2 \right) \ch_1^{B + \beta H}(E).
$$
\end{defi}
\begin{rmk}
\label{def:Lambda-main-choice}
\rm 
In the next sections we are mostly interested in the following choice for $\Lambda$:
$$
\Lambda = \frac{c_2(X)}{12} - \frac{c_2(X)\cdot H}{12 H^3} H^2.
$$
Since  $\Lambda \cdot H =0$, we  can write 
\begin{equation}
\label{eq:reduced-BG-ineq-term}
D^{B,\xi}_{\alpha, \beta}(E) = \ch_3^{B + \beta H}(E) +\Lambda \ch_1^{B}(E) -  \left(\xi + \frac{1}{6}\alpha^2 \right)  H^2 \ch_1^{B + \beta H}(E)  .
\end{equation}
Moreover, $D^{B, \xi}_{\alpha, \beta} = D^{B+ \beta H,\xi}_{\alpha, 0}$. 
\end{rmk}

\begin{conj}
\label{conj:BG-ineq-single}
Let us fix classes $H \in \NS(X)$, $B \in \NS_{\RR}(X)$ such that 
$H$ is ample. 
Then there exist
$\Lambda \in H^4(X, \QQ)$ satisfying $\Lambda \cdot H =0$, and
for any   $\alpha \in \mathbb{R}_{>0}$, $\beta \in \mathbb{R}$,  there is a minimal constant 
$$
\xi(\alpha,\beta) \in \RR_{\ge 0}
$$
such that  all tilt slope $\nu_{H, B+ \beta H, \alpha}$-stable objects $E \in \bB_{ H, B+ \beta H}$
with $\nu_{H, B+ \beta H, \alpha}(E) =0$ satisfy the inequality:
\begin{equation*}
D^{B,\xi(\alpha, \beta)}_{\alpha, \beta}(E) \le 0.
\end{equation*}
Hence, for any $\xi \ge \xi (\alpha,\beta)$, we have $D^{B,\xi}_{\alpha, \beta}(E) \le 0$.
\end{conj}

\begin{rmk}
\label{def:xi(A)}
\rm
Let 
$$
A : B + \RR \langle H \rangle  \to \RR_{\ge 0}
$$
 be a continuous function. Assume Conjecture \ref{conj:BG-ineq-single} holds 
for all $\alpha \in \mathbb{R}_{>0}, \beta \in \mathbb{R}$ such that $ \alpha \ge A(B + \beta H)$. 
Then we can define the following  non-negative constant
$$
\xi(A) = \max \{ \xi (\alpha, \beta) : \alpha \in \mathbb{R}_{>0}, \beta \in \mathbb{R}, \alpha \ge A(B +\beta H) \}.
$$
\end{rmk}
Therefore, in addition to Conjecture \ref{conj:BG-ineq-single}, we can conjecture the following for a family of stability parameters.
\begin{conj}
\label{conj:BG-ineq}
Let us fix classes $H \in \NS(X)$, $B \in \NS_{\RR}(X)$ such that 
$H$ is ample. Let $A : B + \RR \langle H \rangle  \to \RR_{\ge 0}$ be a continuous function.
There exist
$\Lambda \in H^4(X, \QQ)$ satisfying $\Lambda \cdot H =0$, and
a constant $\xi({A})  \in \RR_{\ge 0}$ such that 
for any   $\alpha \in \mathbb{R}_{>0}$, $\beta \in \mathbb{R}$ 
with $\alpha \ge A(B +\beta H)$,
all tilt slope $\nu_{H, B+ \beta H, \alpha}$-stable objects $E \in \bB_{H, B+ \beta H}$
with $\nu_{H, B+ \beta H, \alpha}(E) =0$ satisfy the inequality:
\begin{equation*}
D^{B,\xi(A)}_{\alpha, \beta}(E) \le 0.
\end{equation*}
\end{conj}

\begin{rmk}
\rm
This modified conjectural inequality coincides with  Bogomolov-Gieseker type inequality in \cite{BMT} 
when 
\begin{equation}
\label{eqn:usual-BG-condition}
A = 0, \ \Lambda =0, \ \text{ and } \  \xi(A) =0.
\end{equation}
In this paper, we are mostly interested in the class $\Lambda$ as defined in Remark \ref{def:Lambda-main-choice}.
Many 3-folds where the Bogomolov-Gieseker type inequality conjecture  in \cite{BMT} holds satisfy \eqref{eqn:usual-BG-condition}
  for $\Lambda$ defined  in Remark \ref{def:Lambda-main-choice}.   For example,  when $X$ is an abelian 3-fold 
  ($c_2(X) =0$), or a Fano 3-fold with Picard rank one ($c_2(X)$ is proportional to $H^2$). 
\end{rmk}

\begin{rmk}
\rm
The case when $A =0$ gives the modified  Bogomolov-Gieseker type inequality conjecture  in \cite{BMSZ} for Fano 3-folds. We consider details of this case in Section \ref{sec:Fano3-A=0}.
\end{rmk}

\begin{note}
\rm
Suppose Conjecture \ref{conj:BG-ineq} holds on $X$ with respect to some continuous function
 $A : B + \RR \langle H \rangle  \to \RR_{\ge 0}$. Since tilt stability is preserved under the tensoring by line bundles and dualizing (see Proposition \ref{prop:stability-B-dual}), optimal $A$ should satisfy the following periodic property
\begin{equation*}
A(B + (\beta +1)H)=  A(B + \beta H),
\end{equation*}
and when $B \in \RR \langle H \rangle$, $A(-B -\beta H)=A(B + \beta H)$.
\end{note}

The following is a straightforward expectation from the formulation of the modified Bogomolov-Gieseker type   inequalities.
\begin{note}
\rm
Suppose Conjecture \ref{conj:BG-ineq} holds on $X$ with respect to two continuous functions $A, A' : B + \RR \langle H \rangle  \to \RR_{\ge 0}$, such that $A'(B+ \beta H) \ge A(B +\beta H)$ for all $\beta \in \RR$.
Then for minimal possible $\xi(A)$ and $\xi(A')$, we have 
$$
\xi (A') \le \xi (A).
$$
\end{note}

In particular, we conjecture the following:
\begin{conj}
\label{conj:compareAforBGineq}
Suppose Conjecture \ref{conj:BG-ineq} holds for $X$ with respect to some continuous function $A : B + \RR \langle H \rangle  \to \RR_{\ge 0}$ having $\xi(A)>0$. Then
there exists a continuous function $A' : B + \RR \langle H \rangle  \to \RR_{\ge 0}$, such that
Conjecture \ref{conj:BG-ineq} holds with respect to $A'$ with
$$
\xi(A') =0.
$$
\end{conj}
\begin{note}
\rm
We verify the above conjecture for the blow-up of $\PP^3$ at a point in Section \ref{sec:blowup-P3}. 
\end{note}

This modification of the conjectural inequalities does not affect the corresponding constructions of Bridgeland stability conditions. In particular, similar to \cite[Lemma 8.3]{BMS} we have the following:

\begin{thm}
\label{thm:stab-family}
If   Conjecture \ref{conj:BG-ineq-single} holds for $X$ with respect to some $\alpha, \beta$ then the pair
$$
\left(\aA_{ H, B +\beta H, \alpha}, Z_{ H, B+ \beta H,  \alpha}^{a,b}\right)
$$
 defines a Bridgeland stability condition on $X$. Here $\aA_{H, B +\beta H, \alpha}$ is the heart of a bounded t-structure as constructed in \eqref{def:double-tilt} of Section \ref{subsec:tiltstab},  and 
\begin{align*}
Z_{ H, B+ \beta H, \alpha}^{a,b} = \left( -\ch_3^{B+\beta H} + bH \ch_2^{B+\beta H} + \left( - \Lambda + aH^2 \right) \ch_1^{B+\beta H} \right)
+  \sqrt{-1}  \left( H \ch_2^{B+\beta H} - \frac{\alpha^2}{2} H^3 \ch_0 \right)
\end{align*}
with $a,b \in \mathbb{R}$ satisfying $a >  \xi(\alpha, \beta) + (\alpha^2/6) + (\alpha |b|/2)$. 
\end{thm}

\subsection{Equivalent form of the conjecture}
\label{subsec:betabar-restriction}
In this subsection we formulate an equivalent form of Conjecture \ref{conj:BG-ineq} which only considers the modified Bogomolov-Gieseker  type inequalities for  a small class of tilt stable objects. 
This can be considered as a modification of \cite[Conjecture 5.3]{BMS} and in the next subsection we show that it is equivalent to Conjecture \ref{conj:BG-ineq}. 
We adapt some methods   from \cite[Section 5]{BMS} and \cite{Mac}. 

Let us consider the $\nu_{ H, B+ \beta H, \alpha}$ tilt stability parametrized by $\alpha \in \mathbb{R}_{>0}$ and $\beta \in \mathbb{R}$. 
By definition  
$$
\nu_{ H, B+ \beta H, \alpha} = \frac{H \ch_2^{B+ \beta H}  - ({\alpha^2}/{2}) H^3 \ch_0 }{H^2 \ch_1^{B+ \beta H}}.
$$
Hence, we  consider 
$$
Z_{\alpha, \beta}^{\nu} = -\left( H \ch_2^{B+\beta H}  - \frac{\alpha^2}{2} H^3 \ch_0  \right) +\sqrt{-1} \, H^2 \ch_1^{B+\beta H}
$$
as the associated group homomorphism, more precisely, the weak stability function as introduced in \cite{PT} of the corresponding tilt stability.

In the rest of this section, let us fix some $\alpha_0 \in \mathbb{R}_{>0}$.  

\begin{defi}
Let $E$ be an object in $\bB_{H , B}$  with $\nu_{H,B, \alpha_0}(E) =0$.
\begin{equation*}
C(E) = \left\{ (\beta, \alpha): \ \Ree Z_{\alpha, \beta}^{\nu}  (E) = H \ch_2^{B+\beta H}(E)  - \frac{\alpha^2}{2} H^3 \ch_0(E)  = 0, \text{ and } 0 \le \alpha \le \alpha_0 \right\}.
\end{equation*}
Hence,  $(0, \alpha_0) \in C(E)$.  See Figures \ref{fig:C(E)whench_0>0} and \ref{fig:C(E)whench_0<0}.
\end{defi}

\begin{minipage}{0.50\linewidth}
\begin{center}
\begin{tikzpicture}
      \draw[->] (-1,0) -- (4,0) node[right] {$\beta$};
      \draw[->] (0,-1) -- (0,4) node[above] {$\alpha$};
      \draw[scale=0.5,domain=0:5,smooth,variable=\x,black] plot ({\x},{(sqrt((49)/(40))*sqrt((\x-5)*(\x-8))});
           \draw (0, 3.5 ) node[left] {$(0, \alpha_0)$};
      \draw (0, 3.5 ) node {$\bullet$};
        \draw (1.5, 1.79571 ) node[right] { $C(E)$};
    \end{tikzpicture}
        \captionof{figure}{$C(E)$ when $\ch_0(E)> 0$}
        \label{fig:C(E)whench_0>0}
\end{center}
\end{minipage}
\begin{minipage}{0.50\linewidth}
\begin{center}
\begin{tikzpicture}
      \draw[->] (-4,0) -- (1,0) node[right] {$\beta$};
      \draw[->] (0,-1) -- (0,4) node[above] {$\alpha$};
      \draw[scale=0.5,domain=-5:0,smooth,variable=\x,black] plot ({\x},{(sqrt((49)/(40))*sqrt((-\x-5)*(-\x-8))});
     \draw (0, 3.5 ) node[right] {$(0, \alpha_0)$};
      \draw (0, 3.5 ) node {$\bullet$};
            \draw (-1.5, 1.79571 ) node[left] { $C(E)$};
    \end{tikzpicture}
        \captionof{figure}{$C(E)$ when  $\ch_0(E) < 0$}
        \label{fig:C(E)whench_0<0}
\end{center}
\end{minipage} 
\\

Let $A : B + \RR \langle H \rangle  \to \RR_{\ge 0}$ be a continuous function. 
For a given object $E$, if we have
$$
\lim_{\alpha \to A(B + \beta H)^+} - \Ree Z_{\alpha , \beta}^\nu(E) =0,
$$ 
when  
$\beta \to \overline{\beta}$, then $\overline{\beta}$  satisfies 
\begin{equation}
\label{eqn:betabar-relation}
H \ch_2^{B+\overline{\beta} H }(E) - \frac{(A(B+\overline{\beta} H))^2}{2}H^3 \ch_0(E) = 0.
\end{equation}
That is, 
$\obe^2 (H^3 \ch_0(E))- 2 \obe (H^2 \ch_1^B(E)) - (A(B + \obe H))^2 (H^3 \ch_0(E)) + 2 H\ch_2^B = 0$.
\begin{defi}
\label{defn:betabar}
We define $\obe_A(E) $ to be the set of roots $\obe$ of \eqref{eqn:betabar-relation}; so that 
  $( \obe , A(B + \obe H)) \in C(E)$ for each $\obe \in \obe_A(E)$. See Figure \ref{fig:betabar(E)}.
\end{defi}

\begin{center}
\begin{tikzpicture}
      \draw[->] (-1,0) -- (4,0) node[right] {$\beta$};
      \draw[->] (0,-1) -- (0,4) node[above] {$\alpha$};
      \draw[scale=0.5,domain=0:5,smooth,variable=\x,black] plot ({\x},{(sqrt((49)/(40))*sqrt((\x-5)*(\x-8))});
           \draw (0, 3.5 ) node[left] {$(0, \alpha_0)$};
      \draw (0, 3.5 ) node {$\bullet$};
       \draw[scale=0.5,domain=-2:9,smooth,variable=\x,black] plot ({\x},{3 +  sin(50*\x)});
               \draw (4.5, 2 ) node[right] { $\alpha = A(B + \beta H)$};
       \draw (0.5, 2.6 ) node[above right] { $C(E)$};
         \draw (1.5, 1.79571 ) node[right] { $(\obe, \oA)$};
      \draw  (1.5, 1.79571  ) node {$\bullet$};
    \end{tikzpicture}
        \captionof{figure}{$\obe \in \obe_A(E) $ such that $(\obe, \oA) \in C(E)$}
        \label{fig:betabar(E)}
\end{center}

\begin{exam}
\rm
Unlike the case in Figure \ref{fig:betabar(E)},  the set $\obe_A(E)$ can have many points. The following  is such an example, and it appears in Section \ref{sec:blowup-P3}.
For any $m \in \ZZ$,  $\oO_{X}(m H)$ and $\oO_X(mH)[1]$ are tilt stable. 
Let us consider the  continuous function  
$A: \RR \langle H \rangle \to \RR_{\ge 0}$,
defined by, for $\beta \in [-1/2, 0)$, $A(\beta H) = 1+ \beta$;
 for $\beta \in [0,1/2)$, $A(\beta H ) = 1- \beta$; together with the relation
 $A((\beta +1)H) = A(\beta H)$.
 One can check that 
 $$
 \obe_A(\oO_X(mH)) = [m-1, \, m-(1/2)].
 $$
See Figure \ref{fig:betabar-multi-values}, for  $\obe_A(\oO_X(2H))$.
\begin{center}
\begin{tikzpicture}
      \draw[->] (-2.5,0) -- (5.5,0) node[right] {$\beta$};
      \draw[->] (0,-0.5) -- (0,5) node[above] {$\alpha$};
            \draw[scale=2,domain=-1:-0.5,smooth,variable=\x,black] plot ({\x},{(-\x)});
            \draw[scale=2,domain=-0.5:0,smooth,variable=\x,black] plot ({\x},{(1+\x)});
            \draw[scale=2,domain=0:0.5,smooth,variable=\x,black] plot ({\x},{(1-\x)});
              \draw[scale=2,domain=0.5:1.0,smooth,variable=\x,black] plot ({\x},{(\x)});    
              \draw[scale=2,domain=1:1.5,smooth,variable=\x,black] plot ({\x},{(2- \x)});   
              \draw[scale=2,domain=1.5:2.0,smooth,variable=\x,black] plot ({\x},{(-1+\x)});   
                \draw[scale=2,domain=2:2.5,smooth,variable=\x,black] plot ({\x},{(3- \x)});   

              \draw[scale=2,domain=0:2,smooth,variable=\x,black] plot ({\x},{(2- \x)}); 
              
    \draw [dashed]  (-2.5,1)  -- (0,1) node [below left] {$\frac{1}{2}$}
       -- (5.5,1);
    \draw [dashed]  (1,1)  -- (1,0) node [below] {$\frac{1}{2}$};
        \draw [dashed]  (2,2)  -- (2,0) node [below] {$1$};
            \draw [dashed]  (1,1)  -- (1,0) node [below] {$\frac{1}{2}$};
                        \draw [dashed]  (3,1)  -- (3,0) node [below] {$\frac{3}{2}$};

            \draw  (0,0)  node [below left] {$0$};
              \draw  (0,2)  node [above left] {$1$};
                            \draw  (4,0)  node [below] {$2$};
                                     \draw  (0,4)  node [above left] {$2$};
                     \draw  (1,3)  node [right] {$C(\oO(2H))$};
                                                 \draw  (4.5,1.5)  node [right] {$\alpha = A(\beta H)$};

    \end{tikzpicture}
        \captionof{figure}{$\obe_A(\oO_X(2H)) =[1, \, 3/2]$ }
        \label{fig:betabar-multi-values}
\end{center}
\end{exam}

We need the following definition extending the similar notion in \cite{Li}.
\begin{defi}
\label{defn:betabar-stability}
An object     $E \in D^b(X)$ is called $\obe_{A}$-stable if  for any $\obe \in \obe_A(E)$ there is an open neighbourhood $U \subset \RR_{\beta,\alpha}^2$ containing $(\obe, A(B+ \obe H))$ such that for any $(\beta, \alpha) \in U$ with $\alpha>0$,
$E \in \bB_{H, B+\beta H}$ is $\nu_{H, B + \beta H,  \alpha}$-stable. 
\end{defi}

\begin{rmk}\rm
When $A=0$, the above notion of $\obe_A$ stability is exactly the same notion of $\obe$ stability  in \cite{Li}.
\end{rmk}

From the definition of $\obe_A$-stability and Proposition \ref{prop:stability-B-dual}, we have the following:
\begin{prop}
\label{prop:betabar-stab-dual}
Let $E$ be an object in $D^b(X)$. Then $E $ is $\obe_{A}$-stable with respect to
the stability parameters $B \in \NS_{\RR}(X)$,  and 
 some continuous function 
$A : B + \RR \langle H \rangle  \to \RR_{\ge 0}$ if and only if 
$E^1 = H^{1}_{\bB_{H,-B}}(E^\vee)$ is $\obe_{\widehat{A}}$-stable 
with respect to
the stability parameters $- B \in \NS_{\RR}(X)$,  and the
 continuous function 
$\widehat{A} : -B + \RR \langle H \rangle  \to \RR_{\ge 0}$ defined by 
$\widehat{A}(-B -\beta H) = A(B + \beta H)$.
\end{prop}
For a small class of tilt stable objects, Conjecture \ref{conj:BG-ineq} reads as follows:
\begin{conj}
\label{conj:limitBG}
Let us fix classes $H \in \NS(X)$, $B \in \NS_{\RR}(X)$ such that 
$H$ is ample. Let $A : B + \RR \langle H \rangle  \to \RR_{\ge 0}$ be a continuous function.
There exist
$\Lambda \in H^4(X, \QQ)$ satisfying $\Lambda \cdot H =0$, and
a constant $\xi({A})  \in \RR_{\ge 0}$ such that 
any $\obe_{A}$-stable object $E \in D^b(X)$ satisfies the inequality 
$$
D^{B,\xi(A)}_{A(B +\obe H ) ,  \obe}(E) \le 0, \ \text{ for each } \obe \in \obe_A(E).
$$ 
\end{conj}

The following is the key theorem for us.
\begin{thm}
\label{thm:conj-equivalence}
Conjectures \ref{conj:BG-ineq} and \ref{conj:limitBG} are equivalent. 
\end{thm}

\subsection{Proof of the equivalences of the conjectures} 
We need  few results to prove Theorem \ref{thm:conj-equivalence}. 

Let   $\xi \in \RR_{\ge 0}$ be some fixed constant. 

\begin{lem}
\label{prop:D-derivative-wrt-alpha}
Let $E$ be an object in $\bB_{ H , B}$  with $\nu_{H, B, \alpha_0}(E) =0$.
Then  along $C(E)$ we have
$$
\frac{d}{d\alpha}\left( D^{B,\xi}_{\alpha, \beta}(E) \right) = \frac{-\alpha \ \overline{\Delta}_{H, B}(E) - 3\xi (H^3 \ch_0(E))^2}{3H^2 \ch_1^{B+\beta H}(E)}.
$$
\end{lem}
\begin{proof}
For $( \beta, \alpha) \in C(E)$,   we have $H \ch_2^{B+\beta H}(E)  - (\alpha^2/2) H^3 \ch_0(E) = 0$. By differentiating both sides with respect to 
$\alpha$ we get
\begin{equation}
\label{eqn:dbeta-over-dalpha}
\frac{d\beta}{d \alpha} = - \frac{\alpha H^3 \ch_0(E)}{H^2 \ch_1^{B+\beta H}(E)}.
\end{equation}
By differentiating the expression of $D^{B,\xi}_{\alpha, \beta}(E)$ in \eqref{eq:reduced-BG-ineq-term} with respect to $\alpha$, we get
$$
\frac{d}{d\alpha}\left( D^{B,\xi}_{\alpha, \beta}(E) \right) = - H \ch_2^{B+\beta H}(E) \, \frac{d\beta}{d \alpha}  - \frac{\alpha}{3}H^2 \ch_1^{B+\beta H}(E) + 
\left(\xi + \frac{\alpha^2}{6}\right)H^3 \ch_0(E) \, \frac{d\beta}{d \alpha}.
$$
Since $ H \ch_2^{B+\beta H}(E)  =  (\alpha^2/2) H^3 \ch_0 (E)$ and by substituting the expression of ${d\beta}/{d \alpha}$, we obtain the required expression. 
\end{proof}

\begin{note}
\label{prop:alpha-beta-curve}
\rm
Let $E$ be an object satisfying the conditions in above lemma. 
So $H \ch^B_2(E) = (\alpha_0^2/2) H^3 \ch_0(E)$, and for $(\beta, \alpha) \in C(E)$ we have
$$
 \left(\beta^2 - \alpha^2\right) H^3 \ch_0(E) - 2 \beta H^2 \ch_1^B(E) + \alpha_0^2 H^3 \ch_0(E) =0.
$$
Moreover, by Proposition \ref{prop:muslopeboundfortiltstability} we have   
$$
\Del_{H,B}(E) = (H^2 \ch_1^B(E))^2 - \alpha_0^2 (H^3 \ch_0(E))^2 \ge 0.
$$
When $\ch_0(E) =0$, $C(E)$ is a vertical line at $\beta =0$ from $\alpha =0$ to $\alpha_0$ in $(\beta, \alpha)$-plane. 

Let us consider the case $\ch_0(E) \ne 0$. By \eqref{eqn:dbeta-over-dalpha} in Lemma \ref{prop:D-derivative-wrt-alpha}, along $C(E)$ at $(\beta, \alpha)$ we have 
\begin{align*}
\left(\frac{d\alpha}{d \beta}\right)^2   & = \left( \frac{H^2\ch_1^B(E)}{\alpha H^3 \ch_0(E)}\right)^2
= \frac{\Del_{H,B}(E)}{\alpha^2 (H^3 \ch_0(E))^2}+ 1 \\
& \ge  \frac{\Del_{H,B}(E)}{\alpha_0^2 (H^3 \ch_0(E))^2} + 1 \ge 1.
\end{align*}
\end{note}

\begin{prop}
\label{prop:tilt-stable-obj-same-B-category}
Let $E \in \bB_{H, B}$ be a tilt stable object with $\nu_{H, B, \alpha_0}(E) =0$. 
Then $E \in \bB_{ H, B + \beta H}$ for $\beta \in [-\alpha_0, \alpha_0]$; in particular
$E \in \bB_{ H, B + \beta H} $ for all $(\alpha, \beta) \in C(E)$. 
\end{prop}
\begin{proof}
From Proposition \ref{prop:slope-bounds}, 
we have $E \in \bB_{ H, B + \beta H}$ for all $\beta \in [-\alpha_0, \alpha_0]$. 
In particular,  from the discussion in Note \ref{prop:alpha-beta-curve}, for any  $( \beta, \alpha)$ on $C(E)$.
\end{proof}

Let us prove the key theorem. 
\begin{proof}[Proof of Theorem \ref{thm:conj-equivalence}] 
One implication in the theorem is obvious. Let us prove the other implication using contradiction method.  

Assume Conjecture \ref{conj:limitBG} holds for our  3-fold $X$, and  there is a counterexample for Conjecture \ref{conj:BG-ineq}.
Let $E \in \bB_{H,B}$ be a $\nu_{H,B, \alpha_0}$ tilt stable object with $\nu_{H,B, \alpha_0}(E) =0$. 
By deforming tilt stability parameters appropriately  in a small neighbourhood,  we can assume $B$ is a rational class.

Suppose  $D^{B,\xi}_{\alpha_0, 0}(E) > 0$ for a contradiction. 
By Proposition \ref{prop:tilt-stable-obj-same-B-category}, $E$ stays in the same tilt category for all $(\beta, \alpha )$ in $C(E)$.

Let us consider the tilt stability of $E$ along $C(E)$ when $\alpha$ is decreasing from $P_0 = (0, \alpha_0)$. 

\subsection*{Notation}
For a sequence of pairs $P_j = (\beta_j, \alpha_j)$, $j \ge 0$ in $\mathbb{R}^2 $  we simply write 
\begin{align*}
\bB_{P_j} = & \bB_{H, B+ (\beta_1 + \cdots + \beta_j)H}, \\
\nu_{P_j} = & \nu_{H, B+ (\beta_1 + \cdots + \beta_j)H, \alpha_j } .
\end{align*}

By Proposition \ref{prop:muslopeboundfortiltstability}, $\overline{\Delta}_{H,B}(E) \ge 0$. 
When $\overline{\Delta}_{H,B}(E) >0$,   there might be a  point $P_1 = (\beta_1, \alpha_1) \in C(E)$  such that $ E \in \bB_{P_1}$ becomes strictly $\nu_{P_1}$-semistable.
From Lemma \ref{prop:D-derivative-wrt-alpha}, we have 
$$
0 < D^{B,\xi}_{\alpha_0, 0}(E) < D^{B, \xi}_{\alpha_1, \beta_1}(E) = D^{B + \beta_1 H, \xi}_{\alpha_1, 0}(E).
$$
From the Jordan-H\"older filtration of $E$, there exists $\nu_{P_1}$-stable factor $E_1 \in \bB_{P_1}$  of $E$ with $ D^{B + \beta_1 H, \xi}_{\alpha_1, 0}(E_1)>0$.
Moreover, from Proposition \ref{prop:discrimi-ses-decrease}
$$
\overline{\Delta}_{H,B}(E) > \overline{\Delta}_{H,B}(E_1).
$$
Now we take $E_1 \in \bB_{P_1}$ and consider the tilt stability along $C(E_1)$ in $\alpha$ decreasing direction from $(0, \alpha_1) \in C(E_1)$. 
In this way there exists a sequence of points $P_j =(\beta_j, \alpha_j) \in C(E_{j-1})$ with
$$
\alpha_0 > \alpha_1 > \alpha_2 > \cdots > \alpha_j > \cdots 
$$
$$
 D_{\alpha_j ,0}^{B+ (\beta_1 +\cdots + \beta_j)H , \xi }(E_j) > 0  \  \text{for all } j, \ \text{ and}
$$
$$
\overline{\Delta}_{H,B}(E) > \overline{\Delta}_{H,B}(E_1) > \cdots >\overline{\Delta}_{H,B}(E_j) > \cdots  \ge 0.
$$ 

Since $B$ is chosen to be rational, the image of $\overline{\Delta}_{H,B}$ forms a discrete set in $\mathbb{R}$; 
hence, this sequence  terminates. That is there exists $E_j \in \bB_{P_j}$ which is $\nu_{P_j}$-stable, with 
$ D_{\alpha_j ,0}^{B+ (\beta_1 +\cdots + \beta_j)H , \xi}(E_j) > 0$, and 
\begin{enumerate}
\item either $\overline{\Delta}_{H,B} (E_j) = 0$, 
\item  or $E_j$ is $ \nu_{H, B+ (\beta_1 + \cdots + \beta_j + \beta)H, \alpha } $-stable for all $(\beta, \alpha) \in C(E_j)$.
\end{enumerate} 
From Propositions \ref{prop:muslopeboundfortiltstability} and \ref{prop:Delta0mustabletiltstable}, in case (i) we have $E_j$ is $ \nu_{ H, B+ (\beta_1 + \cdots + \beta_j + \beta)H, \alpha}$-stable for all $(\beta, \alpha) \in C(E_j)$. 
From Lemma \ref{prop:D-derivative-wrt-alpha}, we have 
$$
0 <  D_{\alpha_j ,0}^{B+ (\beta_1 +\cdots + \beta_j)H ,\xi}(E_j) 
\le  D_{\overline{A} ,\obe}^{B+ (\beta_1 +\cdots + \beta_j)H ,\xi}(E_j),
$$
where $(\obe, \overline{A}) \in C(E_j)$, such that  $\overline{A} = A({B+ (\beta_1 +\cdots + \beta_j + \obe)H})$; that is $\obe \in \obe_A(E_j)$.
But this is  not possible as we already assume Conjecture \ref{conj:limitBG} holds for $X$.  This completes the proof.
\end{proof}

\section{Some $\Hom$ Vanishing Results for $\obe_A$ Stable Objects}
\label{sec:Hom-vanishing}
We follow the same notation in Section \ref{sec:BGconj} for our smooth projective 3-fold $X$. 
Let $H \in \NS(X)$ be an ample divisor class. 
Let $B$ be a class proportional to $H$. 

We have the following vanishing results for $\obe_{A}$-stable objects.

\begin{prop}
\label{prop:betabar-hom-vanishing}
Let $A: B + \RR \langle H \rangle \to \RR_{\ge 0}$ be a continuous function. 
Let $E \in D^b(X)$ be a $\obe_{A}$-stable object.
Let $(\obe, \overline{A}) \in C(E)$, such that 
$\overline{A} = A(\obe H)$.
In other words, 
there is a small  neighbourhoud $U \subset \RR_{\beta, \alpha}^2$ containing $(\obe, \oA)$, 
such that for any $(\beta, \alpha) \in U$ with $\alpha>0$, 
$E_{H,\beta H}$ is $\nu_{H, \beta H, \alpha}$ tilt stable, satisfying 
$H \ch_2^{\obe H}(E) - (\oA^2/2)H^3 \ch_0(E)=0$.
Suppose
 $H^2 \ch_1^{\obe H}(E) >0$. 
 For any  $k \in \mathbb{Z}$ we have the following:
 \begin{enumerate}[leftmargin=*]
 \item If $k < \obe - \overline{A}$ then for all $j \le 0$
 $$
 \Hom_{\SX}(E, \oO_X(kH)[1+j]) =0.
 $$
  \item 
If $k = \obe - \overline{A}$, with $\Del_H(E) >0$, and $\overline{A} >0$, then  for all $j \le 0$
$$
\Hom_{\SX}(E, \oO_X(k H)[1+j]) =0.
$$
 \item If $k > \obe+ \overline{A}$ then for all $j \le 0$
 $$
 \Hom_{\SX}(\oO_X(kH), E[j]) =0.
 $$
 \item 
If $k = \obe + \overline{A}$, with $\Del_H(E) >0$, and $\overline{A} >0$, then  for all $j \le 0$
$$
\Hom_{\SX}(\oO_X(k H) , E[j]) =0.
$$
 \end{enumerate}
\end{prop}
\begin{proof}
(i) \ Let $k$ be an integer such that $ k <  \obe - \overline{A}$.
Since  $E, \oO_X(kH)[1] \in \bB_{H, \obe H}$, we have 
$\Hom_{\SX}(E, \oO_X(kH)[1+j]) =0$ for all $j \le -1$. 
Let us prove the $\Hom$ vanishing for $j=0$ case. Let 
 $$
  0 < \vep < (\obe   - \overline{A}  - k)/2.
$$  
By Proposition \ref{prop:Delta0mustabletiltstable}
$$
\oO_X(kH)[1] \in \bB_{ H, (\obe - \vep) H}
$$
 is $\nu_{H,  (\obe - \vep) H,  \overline{A}+ \vep}$-stable  with 
\begin{align*}
\nu_{H,  (\obe - \vep) H,  (\overline{A}+ \vep)}(\oO_X(kH)[1]) & 
  = - \frac{(\obe - \overline{A} -k -2\vep)(\obe + \overline{A} - k)}{(\obe - k -\vep) } < 0.
\end{align*}
Since $H \ch_2^{\obe H}(E) - (\overline{A}^2/2) H^3 \ch_0(E)= 0$,
\begin{align*}
\nu_{H,  (\obe - \vep) H,  (\overline{A}+ \vep)}(E) & = \frac{\vep H^2 \ch_1^{(\obe -\overline{A})H}(E)}{ H^2 \ch_1^{(\obe -\vep)H}(E)}.
\end{align*}
Since $E$ is $\obe_{A}$-stable with $H^2 \ch_1^{\obe H}(E) > 0$, so for small enough $\vep >0$, $H^2 \ch_1^{( \obe - \vep)H}(E) > 0$. 
Also by (ii) of Proposition \ref{prop:muslopeboundfortiltstability},  
$H^2 \ch_1^{( \obe -\overline{A})H}(E) \ge 0$.
Therefore,
$$
\nu_{H,  (\obe - \vep) H,  (\overline{A}+ \vep)}(E) \ge 0,
$$
and hence, we have $\Hom_{\SX}(E, \oO_X(kH)[1]) = 0$  as required. \\

\noindent (ii) \  For $\vep >0$, by direct computation,
\begin{align*}
& \nu_{H, (\obe - \vep)H, (\overline{A} - \vep)}(\oO_X(kH)[1])  = 0,  \ \text{and} \\
& \nu_{H, (\obe - \vep)H, (\overline{A} - \vep)}(E)   = \frac{\vep H^2 \ch_1^{(\obe - \overline{A})H}(E)}{H^2 \ch_1^{(\obe -\vep)H}(E)}.
\end{align*}
From (i) of Proposition \ref{prop:muslopeboundfortiltstability}, we have 
$H^2 \ch_1^{(\obe - \overline{A})H}(E) >0$, and so for small enough $\vep>0$
$$
\nu_{H, (\obe - \vep)H, (\overline{A} - \vep)}(E)  >0.
$$
Therefore, we get the required $\Hom$ vanishings by comparing the tilt slopes of tilt stable objects $\oO_X(kH)[1]$ and $E$ for small enough $\vep>0$. \\

\noindent (iii) \ Let $k$ be an integer such that $ k >  \obe +  \overline{A}$.
Since  $E, \oO_X(kH) \in \bB_{ H, \obe H}$, we have 
$\Hom_{\SX}(\oO_X(kH), E[j]) =0$ for all $j \le -1$. 
Let us prove the vanishing for $j=0$ case. 

From Proposition \ref{prop:betabar-stab-dual}, $E^1 \in D^b(X)$ is $\obe_{\widehat{A}}$-stable
where $\widehat{A}$ is defined by $\widehat{A}(\beta H) = A (-\beta H)$. Hence,   $\obe_{\widehat{A}}(E^1) = - \obe_{A}(E)$.
So from part (i),  for $-k< -\obe -\overline{A}$, we have $\Hom_{\SX}(E^1, \oO_X(-kH)[1]) = 0$.
By Proposition \ref{prop:B-obj-like-sheaves-dual}, $E$ fits into the short exact sequence:
$$
0 \to E \to E^{11} \to E^{23} \to 0
$$
in $\bB_{ H,\obe H}$ with $E^{23} \in \Coh_0(X)$. By applying the functor  $\Hom_{\SX}(O(kH), -)$ we get
$$
\Hom_{\SX}(O(kH), E) \hookrightarrow  \Hom_{\SX}(O(kH),  E^{11}) \cong \Hom_{\SX}(E^1, \oO_X(-kH)[1]) = 0.
$$ 
So we have $\Hom_{\SX}(O(kH), E)=0$ as required. \\

\noindent (iv) \ For $\vep >0$, by direct computation,
\begin{align*}
& \nu_{H, (\obe + \vep)H, (\overline{A}-\vep)}(\oO_X(kH))  = 0,  \ \text{and} \\
& \nu_{H, (\obe + \vep)H, (\overline{A}-\vep)}(E)  = \frac{-\vep H^2 \ch_1^{(\obe + \overline{A})H}(E)}{H^2 \ch_1^{(\obe + \vep)H}(E)}.
\end{align*}
From (i) of Proposition \ref{prop:muslopeboundfortiltstability}, we have 
$H^2 \ch_1^{(\obe + \oA)H}(E) >0$, and so for small enough $\vep>0$
$$
\nu_{H, (\obe + \vep)H, (\overline{A}-\vep)}(E)  >0.
$$
Therefore, we get the required $\Hom$ vanishings by comparing the tilt slopes of tilt stable objects $\oO_X(kH)$ and $E$ for small enough $\vep>0$. 
\end{proof}

\section{Strong Form of Bogomolov-Gieseker Inequality for Tilt Stable Objects}
\label{sec:strong-BG-ineq-for-tiltstable}
\subsection{Formulation of the inequality}
This section discusses a strong form of Bogomolov-Gieseker inequality for tilt stable objects. 
In the earliest preprint \cite{PiyFano3} of this work this generalized formulation appeared, and that was somewhat 
similar to the one appeared in the previous preprints of \cite{BMSZ} by Bernardara-Macr\`i-Schmidt-Zhao.  However, there were some issues in those formulations, and the authors of \cite{BMSZ} fixed the problem in their published work.  In particular, they formed a strong form of Bogomolov-Gieseker inequality
for tilt stable objects generalizing the previous work of Li in \cite{Li}. 
In this section we further  generalize \cite[Theorem 3.1]{BMSZ}.

First we need some notions for Fano 3-folds. 
Suppose $X$ be a Fano 3-fold of index $r$.  So $-K_X = rH$ for some ample divisor class $H$, where 
 $r \in \{1,2,3,4\}$. 
Let $d =H^3$ be the degree of $X$. 
 Let $\rho(X) = \rk \NS(X)$ be its Picard rank.

\begin{nota}
\rm
We use the following notation in the rest of this paper:
\begin{itemize}
 \item $\mu_{H, \beta}= \mu_{H, \beta H}$ and $\mu_{H} = \mu_{H,0}$.
 \item $\bB_{\beta}=\bB_{H, \beta H}$.
 \item $\nu_{\beta, \alpha}=\nu_{H,\beta H,\alpha}$.
 \item We say $E \in \bB_{H, \beta H}$ is $\nu_{H, \beta H, \alpha}$-(semi)stable simply by
 $E$ is tilt (semi)stable with respect to the stability parameter  $(\beta, \alpha)$.
\item $\overline{\Delta}_H = \overline{\Delta}_{H,tH} = (H^2 \ch_1)^2 - 2 H^3 \ch_0 H \ch_2$
 (see Definition \ref{def:discriminant}). 
\end{itemize}
\end{nota} 

\begin{defi}
\label{def:kappa(X)}
If $\rho(X) >1$ then, we define:
\begin{align*}
e_1(X) & = \min \,  \{(H^2D)^2 - H^3 (HD^2) >0 : \ D \in \NS(X) \}, \\
e_2(X) & = \min \,  \{(H^2D)^2 + 1 : \ D \in \NS(X) \text{ is effective} \}, \\
\kappa(X) & = \min\left\{ \frac{e_1(X)}{d^2}, \ \frac{e_2(X)}{d^2}, \  \frac{3}{2rd} \right\}.
\end{align*}
Otherwise, that is, for $\rho(X)=1$  we set
$$
\kappa(X)  =\frac{3}{2rd}.
$$
\end{defi}

\begin{exam}
\rm
\label{ex:blow-up-P3}
Let $X$ be the blowup of $\PP^3$ at a point. Let $f : X \to \PP^3$ be the blow up morphism. 
Let $L = c_1 \left( f^* \oO_{\PP^3}(1) \right)$  and let $E$ be the exceptional divisor class.
We have $L^3 =E^3=1$, and $L^{i} E^j =0$ for $i, j \neq 0$.
Also the group $\NS(X) = \ZZ \langle L,  E \rangle$. 
From the blowup formula, $-K_X = 4L -2E$ and so $X$ is a index $2$ Fano 3-fold. 
Therefore, in the above notation
$$
H = 2L-E.
$$
By direct computation,   the degree of $X$ is 
$d = H^3 = 7$.

Let $D =aL + bE$ for some $a, b \in \ZZ$. Then we have 
$H^2 D =( 4a +b)$, $HD^2 = (2a^2 -b^2)$, and so
$(H^2D)^2 - H^3 (H D^2) = 2(a+2b)^2$.
Hence we have
\begin{align*}
e_1(X) =2, \ e_2(X) =2, \text{ and } \ 3/(2rd) = 1/21.
\end{align*}
Therefore,
$$
\kappa(X) = \frac{2}{49}.
$$
\end{exam}

\begin{exam}
\rm
\label{ex:P2timesP1}
Let us consider the Fano 3-fold $X = \PP^2 \times \PP^1$ which is of index one.
Let $p_1 : X \to \PP^2$, $p_2 : X \to \PP^1$ be the corresponding projections.
Denote $L_1 =  c_1 \left( p_1^* \oO_{\PP^2}(1) \right)$ and 
$L_2 =  c_1 \left( p_2^* \oO_{\PP^1}(1) \right)$.
Then $\NS(X) = \ZZ \langle L_1, L_2 \rangle$, and $L_1^2 L_2 =1$. 
Also $-K_X = H =3L_1  +2L_2$ and the degree of $X$ is $d= H^3=54$.

Let $D =aL_1 + bL_2$ for some $a, b \in \ZZ$. Then we have 
$H^2 D =( 12a +9b)$, $HD^2 = (2a^2+6ab)$, and so
$(H^2D)^2 - H^3 (H D^2) = 9(2a-3b)^2$.
Hence we have
\begin{align*}
e_1(X) =9, \ e_2(X) =9+12+1 =22,  \text{ and } \ 3/(2rd) = 1/36.
\end{align*}
Therefore,
$$
\kappa(X) = \frac{e_1(X)}{d^2} = \frac{1}{324}.
$$
\end{exam}

The aim of the rest of this section is to prove the following, which  generalizes  \cite[Theorem 3.1]{BMSZ}.
\begin{thm}
\label{thm:strong-BG-ineq}
Let $E$ be a $\nu_{\beta_0, \alpha_0}$-tilt stable object with finite tilt slope and non-isomorphic to 
 $\oO_X(mH)[1]$ or $\iI_Z(mH)$ for any  $m \in \ZZ$ and $0$-subscheme $Z \subset X$.
Let us suppose $\ch_0(E) \ne 0$,

\begin{align}
\label{eqn:strong-BG-ineq-condition}
 \text{if } \ch_0(E) >0, \ & \text{there exist no integers between} \\
 &  \beta_0 + \nu_{\beta_0, \alpha_0}(E) + \sqrt{ \left(\nu_{\beta_0, \alpha_0}(E)\right)^2+ \alpha_0^2}, \text{ and } \nonumber \\
&    2\mu_H(E) -  \beta_0 -\nu_{\beta_0, \alpha_0}(E) - \sqrt{ \left(\nu_{\beta_0, \alpha_0}(E)\right)^2+ \alpha_0^2},\nonumber
\end{align}
\begin{align}
\label{eqn:strong-BG-ineq-condition-2}
 \text{if } \ch_0(E) <0, \  & \text{there exist no integers between} \\
 & -\beta_0 - \nu_{\beta_0, \alpha_0}(E) + \sqrt{ \left(\nu_{\beta_0, \alpha_0}(E)\right)^2+ \alpha_0^2}, 
 \text{ and } \nonumber \\
&   -2\mu_H(E) + \beta_0 + \nu_{\beta_0, \alpha_0}(E) - \sqrt{ \left(\nu_{\beta_0, \alpha_0}(E)\right)^2+ \alpha_0^2},\nonumber
\end{align}
and 
\begin{align}
\label{eq:stron-BG-condition3}
2 \mu_H -  \left( \beta_0 + \nu_{\beta_0, \alpha_0}(E)\right) < r. 
\end{align}
Then 
$$
\frac{\overline{\Delta}_{H}(E)}{(H^3\ch_0(E))^2} \ge \kappa(X).
$$
\end{thm}
The proof of this theorem is similar to that of \cite[Theorem 3.1]{BMSZ}, and we discuss it in the next subsections.

\begin{rmk}
\label{rem:relation-with-BMSZ}
\rm
Suppose $E$ be an object as in Theorem \ref{thm:strong-BG-ineq} with $\ch_0(E) >0$.
Let us write  $\tal_0=\sqrt{ \left(\nu_{\beta_0, \alpha_0}(E)\right)^2+ \alpha_0^2}$, and $\tbe_0 =  \beta_0 + \nu_{\beta_0, \alpha_0}(E)$.
We have 
$\Del_H(E) = (H^2 \ch_1^{\tbe_0 H}(E))^2 - \tal_0^2 (H^3\ch_0(E))^2$. 
So 
\begin{align*}
\frac{\Del_H(E)}{(H^3\ch_0(E))^2}  
&=  \frac{H^2 \ch_1^{(\tbe_0 -\tal_0)H}(E) \cdot H^2 \ch_1^{(\tbe_0 + \tal_0)H}(E)}{(H^3 \ch_0(E))^2} \\
& = \frac{(H^2 \ch_1^{(\tbe_0 +\tal_0)H}(E))^2 + 2 \tal_0 H^3 \cdot H^2 \ch_1^{(\tbe_0 + \tal_0)H}(E)}{(H^3 \ch_0(E))^2}.
\end{align*}
Therefore,
\begin{align*}
 \frac{(H^2 \ch_1^{(\tbe_0 +\tal_0)H)}(E))^2}{(H^3 \ch_0(E))^2} 
< {\frac{\Del_H(E)}{(H^3\ch_0(E))^2}}.
\end{align*}
Hence, 
\begin{align*}
 \mu_H - \tbe_0 - \tal_0
 < \sqrt{\frac{\Del_H(E)}{(H^3\ch_0(E))^2}}.
\end{align*}
That is, 
\begin{align*}
 \mu_H -  \sqrt{\frac{\Del_H(E)}{(H^3\ch_0(E))^2}} <  \tbe_0 + \tal_0,  \\
 2\mu_H - (\tbe + \tal_0) < \mu_H +  \sqrt{\frac{\Del_H(E)}{(H^3\ch_0(E))^2}}.
\end{align*}
So we have \eqref{eqn:strong-BG-ineq-condition} of Theorem \ref{thm:strong-BG-ineq} when there are no integers in the interval  
$$
\left(\mu_H -  \sqrt{\frac{\Del_H(E)}{(H^3\ch_0(E))^2}} , \ \mu_H +  \sqrt{\frac{\Del_H(E)}{(H^3\ch_0(E))^2}} \right].
$$
This is the interval that the authors used in the formulation of Theorem 3.1 in \cite{BMSZ}. 
Similarly one can consider the case $\ch_0 (E)<0$. 

Clearly we have $
\kappa(X) \ge \min \left\{ 1/d^2, \ 3/(2rd)\right\}$; 
where the later constant was considered in \cite[Theorem 3.1]{BMSZ}. 
In particular,  for   Examples \ref{ex:blow-up-P3} and \ref{ex:P2timesP1} we have 
$$
\kappa(X) > \min  \left\{ 1/d^2, \ 3/(2rd)\right\}.
$$ 
\end{rmk}

\subsection{Strong form of Bogomolov-Gieseker inequality for   $|\ch_0| =1$ case}
\begin{prop}
\label{prop:strong-BG-rank-1-case}
Let $E$ be a tilt stable object with a finite tilt slope,  $\DelH(E) >0$ and $|\ch_0(E) | =1$.
Then 
$$
\frac{\overline{\Delta}_H(E)}{d^2} \ge \kappa(X).
$$
\end{prop}
\begin{proof}
Let  $E  \in \bB_{\beta_0}$ be a $\nu_{\beta_0, \alpha_0}$ tilt stable object with 
a finite  tilt slope. 
Since tilt stability is preserved under small deformation of the numerical parameters, we can choose 
$$
\beta_0 \in \QQ.
$$
Now consider the stability of $E$ along the line $\beta = \beta_0$ from $\alpha = \alpha_0$ in the $\alpha$ increasing direction on $R_{\beta, \alpha}^2$ plane. 
There might be a point $P = (\beta_0, \alpha_1)$ where $E$ becomes strictly semistable. Let $E_{i}$ be the Jordan-H\"older tilt stable factors of $E$ at $P$. Since $E$ has a finite slope, all $E_i$'s have  finite tilt slopes.

Now consider the tilt stability of each $E_i$ in the $\alpha$ increasing direction from $P$ along $\beta = \beta_0$. So each $E_i$ has Jordan-H\"older tilt stable factors $E_{i,j}$ at a point $P_{i}$. 
In this way we can find a sequence of tilt stable objects 
$\left\{ E_{i,j,k, \ldots} \right\}$,
with finite tilt slopes.
So 
$$
H^2 \ch_1^{ \beta_0 H}(E)  = \sum_{i,j,k, \ldots}  H^2 \ch_1^{ \beta_0 H} ( E_{i,j,k, \ldots} ), \ \text{ and } \ 0< H^2 \ch_1^{ \beta_0 H} ( E_{i,j,k, \ldots} )< \infty.
$$
Since $\beta_0 \in \QQ$, this sequence terminates. That is we have a finite  collection of objects 
$$
\{ F_s \} = \{ E_{i,j,k, \ldots} \}
$$ 
in 
$\bB_{\beta_0}$ which are tilt stable for $\alpha \ge R$ for some finite  $R >0$. 
Moreover, by repeatedly applying Proposition \ref{prop:discrimi-ses-decrease},
we have 
$$
\DelH(E) \ge \sum_s \DelH(F_s),
$$
where the equality holds only when all $\DelH(F_s) =0$. 

If all $F_s$ have  $\ch_i(F_i) \in \ZZ \langle H \rangle$, then  by definition $\DelH(E) = H^3 H(\ch_1(E)^2 - 2 \ch_2(E)) = 2 H^3 H \cdot C$ for some $C \in H_2(X, \ZZ)$. So  $\DelH(E) \ge 2d > 3d/(2r)  \ge d^2 \cdot \kappa(X)$ as required.

Otherwise, there exists at least one $F_s$ with $\ch_1(F_s) \not \in \ZZ \langle H \rangle$, and we have one of the following cases:
\begin{itemize}[leftmargin=0.5cm]
\item 
If $\ch_0(F_s)> 0$ then from Proposition \ref{prop:limittiltstableobjects},
$F_s \cong \hH^{0}(F_s)$ is a slope semistable torsion free sheaf and so
\begin{align*}
\Del_H(E) & \ge \Del_H(F_s)  =  H^3  H\cdot \Delta(F_s) +  (H^2 \ch_1(F_s))^2 - H^3 H\cdot (\ch_1(F_s))^2 \\
& \ge (H^2 \ch_1(F_s))^2 - H^3 H\cdot (\ch_1(F_s))^2 \ge e_1(X) \ge d^2 \cdot \kappa(X),
\end{align*}
as required. 

\item  If $\ch_0(F_s) < 0$ then from Proposition \ref{prop:discrimi-ses-decrease},
$\hH^{-1}(F_s)$ is a slope semistable reflexive sheaf and $\hH^{0}(F_s) \in \Coh_{\le 1}(X)$; so
\begin{align*}
\Del_H(E) & \ge \Del_H(F_s)  =   \Del_H(\hH^{-1}(F_s))   + 2H^3 \ch_0(\hH^{-1}(F_s)) H \ch_2(\hH^{0}(F_s))    \ge   \Del_H(\hH^{-1}(F_s))  \\
& = H^3  H\cdot \Delta(\hH^{-1}(F_s)) + (H^2 \ch_1(\hH^{-1}(F_s)))^2 - H^3 H\cdot (\ch_1(\hH^{-1}(F_s)))^2\\
& \ge (H^2 \ch_1(\hH^{-1}(F_s)))^2 - H^3 H\cdot (\ch_1(\hH^{-1}(F_s)))^2 \ge e_1(X) \ge d^2 \cdot \kappa(X),
\end{align*}
as required. 

\item  If $\ch_0(F_s)= 0$, then from Proposition \ref{prop:limittiltstableobjects}, $F_s$ is a tilt stable sheaf in $ \Coh_{\le 2}(X) \setminus \Coh_{\le 1}(X)$. So from Proposition \ref{prop:discrimi-ses-decrease},
\begin{align*}
\Del_H(E) & \ge \Del_H(F_s) +1  =    (H^2 \ch_1(F_s))^2 + 1 \ge e_2(X)  \ge d^2 \cdot \kappa(X),
\end{align*}
as required. 

\end{itemize}
This completes the proof.
\end{proof}

\subsection{Proof of the strong form of Bogomolov-Gieseker inequality}

Assume  there is a counter example to Theorem \ref{thm:strong-BG-ineq}. From Proposition \ref{prop:strong-BG-rank-1-case}, it has $|\ch_0| \ge 2$. Also from Proposition \ref{prop:stability-B-dual}, we can assume $\ch_0 \ge 2$. 
Let $E \in \bB_{ \beta_0}$ be such a $\nu_{\beta_0, \alpha_0}$ tilt stable object with minimum $\Del_{H}\ge 0$. 
From Lemma \ref{prop:Delta0tiltstableFano3}, $\DelH(E) > 0$. 
Also from Proposition \ref{prop:stability-B-dual}, we can assume 
$$
E \cong E^{11} = H^1_{B_{\beta_0}}  \left(H^1_{B_{-\beta_0}}(E^\vee)\right)^\vee.
$$
So we have
$$
0 < \frac{\overline{\Delta}_{H}(E)}{(H^3\ch_0(E))^2} < \kappa(X).
$$
Recall, Notation \ref{def:Psi}:
$$
\Psi_{ \beta_0, \alpha_0, \lambda}^{+}(E) := \Psi_{H,\beta_0 H, \alpha_0, \lambda}^{+}(E) = \nu_{\beta_0, \alpha_0}(E) + \sqrt{\alpha_0^2 + (\nu_{\beta_0, \alpha_0}(E))^2+ \lambda}.
$$

 
 \begin{prop}
We have $E \cong \hH^0(E)$ and it is $\nu_{\beta_0, \alpha}$ tilt stable for all $\alpha \ge \alpha_0$. In particular, from Proposition \ref{prop:limittiltstableobjects},  $E$ is a $\mu_H$-slope semistable reflexive sheaf.
\end{prop} 
 \begin{proof}
Consider the stability of $E$ along the line $\beta = \beta_0$ from $\alpha = \alpha_0$ in the $\alpha$ increasing direction on $R_{\beta, \alpha}^2$ plane. 
Then there might be a point $P_1 = (\beta_0, \alpha_1)$ where $E$ becomes strictly semistable. 

Let 
$$
E_{i}, \ i =1, \cdots, N,
$$
  be the Jordan-H\"older tilt stable factors of $E$ at $P$. Since $E$ has a finite tilt 
  slope, all $E_i$'s have  finite tilt slopes.
Therefore, $H^2 \ch_1^{\beta_0 H}(E_i) >0$ for all $i$.

There exist $\theta$  for $E$, and $\theta_i$ for each $E_i$  such that 
$$
H^3 \ch_0(E_i) + \sqrt{-1} H^2 \ch_1^{\beta_0 H}(E_i) \in \RR_{>0} e^{\sqrt{-1} \theta_i},
$$
and 
$$
H^3 \ch_0(E) + \sqrt{-1} H^2 \ch_1^{\beta_0 H}(E) \in \RR_{>0} e^{\sqrt{-1} \theta},
$$
satisfying $\theta \in (0, \pi/2)$ and $\theta_i \in (0, \pi)$. 
So we have
\begin{equation}
\label{eqn:sum-of-slope-like-stab-functions}
H^3 \ch_0(E) + \sqrt{-1} H^2 \ch_1^{\beta_0 H}(E) 
 = \sum_i \left( H^3 \ch_0(E_i) + \sqrt{-1} H^2 \ch_1^{\beta_0 H}(E_i)\right).
\end{equation}
There exists an object $E_k$ such that 
$$
0 < \theta_k \le \theta;
$$
because, otherwise  all $\theta_k \in (\theta, \pi)$, and  so from \eqref{eqn:sum-of-slope-like-stab-functions}, $\theta \in  (\theta, \pi)$; but this is not possible. 

That is $\ch_0(E_k) >0$ and 
$0 < \mu_{H, \beta_0}(E_k) \le \mu_{H, \beta_0}(E)$. 
From Proposition \ref{prop:PsiDeltarelation}, we have 
\begin{equation*}
\label{eqn:mu-bound-Psi-Ek}
 \Psi_{\beta_0,\alpha_1,0}^{+}(E_k) \le \mu_{H, \beta_0}(E_k).
\end{equation*}
Here the equality holds when $\Del_H(E_k) =0$, and from Lemma \ref{prop:Delta0tiltstableFano3}
 in this case we have  
$\ch_0(E_k)=1$, $\ch_1(E_k) \in \ZZ\langle H \rangle$, and so  
\begin{equation*}
\mu_H(E_k) = \beta_0 + \Psi_{\beta_0,\alpha_1,0}^{+}(E_k) \in \ZZ.
\end{equation*} 
From Proposition \ref{prop:PsiDeltarelation},
\begin{equation*}
\label{eqn:Psi-change-with-a-increases}
\Psi_{\beta_0,\alpha_0,0}^{+}(E)  
< \Psi_{\beta_0,\alpha_1,0}^{+}(E)
<\mu_{H,\beta_0}(E)
< \Psi_{\beta_0,\alpha_1,\kappa(X)}^{+}(E)  < \Psi_{\beta_0,\alpha_0, \kappa(X)}^{+}(E).
\end{equation*}
Since $\nu_{\beta_0, \alpha_1}(E_k)  = \nu_{\beta_0, \alpha_1}(E)$,  
we have 
\begin{align*}
\beta_0 +  & \Psi_{\beta_0,\alpha_0,0}^{+}(E)  
 < \beta_0 +   \Psi_{\beta_0,\alpha_1,0}^{+}(E)
= \beta_0+ \Psi_{\beta_0,\alpha_1,0}^{+}(E_k) 
\le \mu_{H}(E_k) \\
& \qquad \le \mu_{H}(E)  
<  \beta_0 + \Psi_{\beta_0,\alpha_1,\kappa(X)}^{+}(E) 
 = \beta_0 + \Psi_{\beta_0,\alpha_1,\kappa(X)}^{+}(E_k)
 < \beta_0 + \Psi_{\beta_0,\alpha_0,\kappa(X)}^{+}(E).
\end{align*}
From assumption \eqref{eqn:strong-BG-ineq-condition}, 
 there are no integers in the interval
  $$
  (\beta_0 + \Psi_{\beta_0,\alpha_0,0}^{+}(E), \ \mu_H(E)].
  $$
  So we have $\beta_0 + \Psi_{\beta_0,\alpha_1,0}^{+}(E) \not \in \ZZ$.
That is, $\mu_H(E_k) \not \in \ZZ$, and so
$$
\Del_H(E_k) > 0.
$$

On the other hand, from Proposition \ref{prop:discrimi-ses-decrease},
$$
\DelH(E) \ge \sum_{1 \le i \le N} \DelH (E_i) \ge \DelH(E_k),
$$
where all the equalities hold only when, $ \DelH (E) =0$. 
Since $\DelH(E) , \DelH(E_k) >0$ we have 
$$
\DelH(E) > \DelH(E_k) >0.
$$

Since 
$\beta_0 + \Psi_{\beta_0,\alpha_1,0}^{+}(E_k) 
<  \mu_{H}(E_k)  
< \beta_0 + \Psi_{\beta_0,\alpha_1,\kappa(X)}^{+}(E_k)$,
from Proposition \ref{prop:PsiDeltarelation}, 
$E_k$ is also a tilt stable object which contradicts Theorem \ref{thm:strong-BG-ineq}. 
Since $E$ is chosen with minimal $\DelH$ this is not possible; so $E$ stays
tilt  stable  for all $(\beta_0, \alpha)$, $\alpha \ge \alpha_0$. 
Since $\ch_0(E)>0$, from Proposition \ref{prop:limittiltstableobjects}, 
$$
E \cong \hH^0(E)
$$ 
is a slope semistable torsion free sheaf. Since $E \cong E^{11}$, we have $E \cong E^{**}$; that is $E$ is a reflexive sheaf.
\end{proof}

\begin{defi} 
\label{def:ImZ(E)=0}
For any object $F$ in $D^b(X)$
$$
Z(F)  = \{ (\beta , \alpha) \in \RR^2: H \ch_2^{\beta H}(F) - (\alpha^2/2) H^3 \ch_0(F) = 0 \}.
$$
So $Z(F[1]) = Z(F)$.
\end{defi}
Let  us define
\begin{align*}
\tbe_0 & = \beta_0 + \nu_{\beta_0, \alpha_0}(E), \\
\tal_0 & =\sqrt{ \alpha_0^2 + \left(\nu_{\beta_0, \alpha_0}(E)\right)^2}.
\end{align*}
Then by direct computation one can verify that $(\tbe_0, \tal_0) \in Z(E)$ (also see \eqref{eqn:nu-to-0-stab-para} in Section \ref{sec:tiltproperties}). 

From Proposition \ref{prop:wall-tiltstable}, $E$ is tilt stable, with zero tilt slope,  with respect to  
$(\beta, \alpha) \in Z(E) $ such that 
$\beta \le \tbe_0$ and $\alpha \ge \tal_0$.

Moreover, by solving $H \ch_2^{\beta H}(E) - (\alpha^2 /2)H^3 \ch_0(E) = 0$, we have 
$$
(-\tbe_0 + 2\mu_H (E), \tal_0) \in Z(E). 
$$
Since $\mu_{H, \beta_0}(E) > \Psi_{\beta_0, \alpha_0, 0}^{+}(E) =  (\tbe_0 - \beta_0) + \tal_0 > \tbe_0 - \beta_0$, we have 
$$
\mu_H(E) < -\tbe_0 + 2 \mu_H (E).
$$
As $E$ is a slope semistable reflexive sheaf $E[1] \in \bB_{-\tbe_0 + 2 \mu_H (E)}$. 
\begin{prop}
\label{prop:stability-E[1]}
The object $E[1] \in \bB_{-\tbe_0 + 2 \mu_H (E)}$ is tilt stable with respect to 
$$
(-\tbe_0 + 2 \mu_H (E), \alpha),  \  \text{ for any } \alpha \ge \tal_0.
$$
\end{prop}
\begin{proof}
Since $E$ is a reflexive sheaf, from Proposition \ref{prop:stability-B-dual}, the claim is equivalent to 
$E^* \in \bB_{\tbe_0 - 2 \mu_H (E)}$ is tilt stable with respect to 
$$
(\tbe_0 - 2 \mu_H (E), \alpha), \  \text{ for any } \alpha \ge \tal_0.
$$
Let us denote 
$$
\hbe_0 = \tbe_0 - 2 \mu_H (E), \  \ \hal_0 = \tal_0.
$$

Since 
$E \in \bB_{\beta_0}$ is $\nu_{\beta_0, \alpha}$ tilt stable for all $\alpha \ge \alpha_0$, 
there exists large enough $a>0$ such that $E^* \in \bB_{\hbe_0} $ is 
$\nu_{\hbe_0, \alpha}$ tilt stable 
 for all $\alpha \ge a$. 

Consider the tilt stability of $E^*$ along the line $\beta = \hbe_0$ from $\alpha = a$ in the $\alpha$ decreasing direction on $R_{\beta, \alpha}^2$ plane. 
Assume for a contradiction there is  a point $ (\hbe_0, \hal_1)$ where $E^*$ becomes strictly semistable
for some $\hal_1 > \hal_0$. 
Let 
$$
F_{i}, \ i =1, \cdots, N',
$$
  be the Jordan-H\"older tilt stable factors of $E^*$ with respect to $ (\hbe_0, \hal_1)$. 
  Since $E^*$ has a finite tilt slope, all $F_i$'s have  finite tilt slopes.
Therefore, $H^2 \ch_1^{\hbe_0 H}(F_i) >0$ for all $i$.

There exist $\varphi$ for $E^*$ and $\varphi_i$ for each $F_i$,  such that 
$$
H^3 \ch_0(F_i) + \sqrt{-1} H^2 \ch_1^{\hbe_0 H}(F_i) \in \RR_{>0} e^{\sqrt{-1} \varphi_i},
$$
and 
$$
H^3 \ch_0(E^*) + \sqrt{-1} H^2 \ch_1^{\hbe_0 H}(E^*) \in \RR_{>0} e^{\sqrt{-1} \varphi},
$$
satisfying $\varphi \in (0, \pi/2)$ and $\varphi_i \in (0, \pi)$. 
We have
\begin{equation}
\label{eqn:sum-stab-func-E*}
H^3 \ch_0(E^*) + \sqrt{-1} H^2 \ch_1^{\hbe_0 H}(E^*)
  = \sum_i \left( H^3 \ch_0(F_i) + \sqrt{-1} H^2 \ch_1^{\hbe_0 H}(F_i) \right).
\end{equation}
There exists an object $F_k$ such that 
$$
0 < \varphi_k \le \varphi,
$$
because, otherwise  all $\varphi_k \in (\varphi, \pi)$, and  so from \eqref{eqn:sum-stab-func-E*}, $\varphi \in  (\varphi, \pi)$; but this is not possible. 
That is $\ch_0(F_k) >0$ and 
$0 < \mu_{H, \beta_0}(F_k) \le \mu_{H, \beta_0}(E^*)$.

From Proposition \ref{prop:PsiDeltarelation}, we have 
\begin{equation*}
\Psi_{\hbe_0, \hal_1,0}^{+}(F_k) \le \mu_{H, \hbe_0}(F_k).
\end{equation*}
Here the equality holds when $\Del_H(F_k) =0$, and in this case from Lemma \ref{prop:Delta0tiltstableFano3} we have  
$\ch_0(F_k) =1$, $\ch_1(F_k) \in \ZZ \langle H \rangle$, and so 
\begin{equation*}
\mu_H (F_k) = \hbe_0 +  \Psi_{\hbe_0, \hal_1,0}^{+}(F_k)  \in \ZZ.
\end{equation*} 
From Proposition \ref{prop:PsiDeltarelation},
\begin{align*}
\Psi_{\hbe_0, \hal_0,0}^{+}(E^{*}) 
<  \Psi_{\hbe_0, \hal_1,0}^{+}(E^{*})
 < \mu_{H, \hbe_0}(E^*)
 < \Psi_{\hbe_0, \hal_1,\kappa(X)}^{+}(E^{*})
  < \Psi_{\hbe_0, \hal_0,\kappa(X)}^{+}(E^{*}).
\end{align*}
Here $\hal_0 = \Psi_{\hbe_0, \hal_0,0}^{+}(E^{*}) $ and $\Psi_{\hbe_0, \hal_0,\kappa(X)}^{+}(E^{*})
   = \sqrt{\hal_0^2 + \kappa(X)}$.
   
Since $\nu_{\hbe_0, \hal_1}(F_k)  = \nu_{\hbe_0, \hal_1}(E^*)$ we have 
\begin{align*}
\hbe_0+ \hal_0 
& <  \hbe_0+ \Psi_{\hbe_0, \hal_1,0}^{+}(E^{*})
= \hbe_0+ \Psi_{\hbe_0, \hal_1,0}^{+}(F_k) 
\le \mu_{H}(F_k) \\
&  \le \mu_{H}(E^*)
< \hbe_0 + \Psi_{\hbe_0, \hal_1,\kappa(X)}^{+}(E^{*})
=  \hbe_0 + \Psi_{\hbe_0, \hal_1,\kappa(X)}^{+}(F_k) 
 < \hbe_0 + \sqrt{\hal_0^2 + \kappa(X)}.
\end{align*}
We have 
$ \hbe_0+ \hal_0  = \beta_0 + \nu_{\beta_0,\alpha_0}(E) - 2 \mu_H(E) + \sqrt{(\nu_{\beta_0,\alpha_0}(E))^2 + \alpha_0^2} $, 
and 
$ \mu_H(E^*) = - \mu_H(E)$.
Therefore, from the assumption \eqref{eqn:strong-BG-ineq-condition} in Theorem \ref{thm:strong-BG-ineq},
there are no integers in the interval
$$
(\hbe_0+ \hal_0, \ \mu_H(E^*) ].
$$
So $\mu_H(F_k) \not \in \ZZ$, and 
 hence, $\Del_H(F_k) >0$.

On the other hand, from Proposition \ref{prop:discrimi-ses-decrease},
$$
\DelH(E^*) \ge \sum_{1 \le i \le N'} \DelH (F_i) \ge \DelH(F_k),
$$
where all the equalities hold only when, $ \DelH (E^*) =0$. 
Since $\DelH(E) = \Del(E^*), \DelH(F_k) >0$ we have 
$$
\DelH(E) > \DelH(F_k) >0.
$$

Since 
$\hbe_0 + \Psi_{\hbe_0,\hal_1,0}^{+}(F_k) 
<  \mu_{H}(F_k)  
< \hbe_0 + \Psi_{\hbe_0, \hal_1,\kappa(X)}^{+}(F_k)$,
 from Proposition \ref{prop:PsiDeltarelation}, $F_k$ is also a tilt stable object which contradicts Theorem \ref{thm:strong-BG-ineq}. Since $E$ is chosen with minimal $\DelH$ this is not possible; so $E^*$ stays tilt stable  for all $(\hbe_0, \alpha)$, $\alpha \ge \hal_0 = \tal_0$ as required. 
\end{proof}

\begin{proof}[Proof of Theorem \ref{thm:strong-BG-ineq}]

Let us define
\begin{align*}
&C'  = Z(E) \cap \{ \beta \le \tbe_0 \}, \\
&C'' = Z(E(-rH)[1]) \cap \{\beta \ge 2 \mu_H(E) - \tbe_0 -r \}.
\end{align*}
The objects $E$ and $E(-rH)[1]$ are tilt stable along the paths $C'$ and $C''$ respectively, with zero tilt slopes. 

Assumption   \eqref{eq:stron-BG-condition3} gives
$2 \mu_H(E) - \tbe_0 -r< \tbe_0$.
So $C'$ and $C''$ intersect each other at $Q = (\beta^o , \alpha^o)$. See Figure \ref{fig:C1intesectC2}. 

\begin{center}
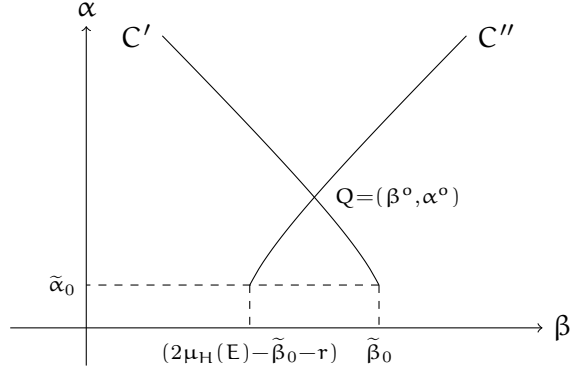

\begin{tikzpicture}
      \draw[->] (-2,0) -- (5,0) node[right] {$\beta$};
      \draw[->] (-1,-0.5) -- (-1,4) node[above] {$\alpha$};
      \draw[scale=1,domain=0:2.85,smooth,variable=\x,black] plot ({\x},{sqrt((\x-5)*(\x-3))});
       \draw (0, 3.87298 ) node[left] {$C'$};
       \draw[scale=1,domain=1.15 :4,smooth,variable=\x,black] plot ({\x},{sqrt((\x-1)*(\x+1))});
        \draw (4, 3.87298 ) node[right] {$C''$};
         \draw  (2,1.732050)   node [right] {\ $\scriptstyle Q = (\beta^o, \alpha^o)$};
     \draw [dashed]  (2.85,0.56789)  -- (-1,0.56789) node [left] {$\scriptstyle \tal_0$};
             \draw [dashed]  (2.85,0.56789)  -- (2.85,0) node [below] {$\scriptstyle \tbe_0$};
              \draw [dashed]  (1.15,0.56789)  -- (1.15,0) node [below] {$\scriptstyle (2 \mu_H(E) -\tbe_0-r)$};
    \end{tikzpicture}
        \captionof{figure}{Intersection of $C'$ and $C''$}
        \label{fig:C1intesectC2}
\end{center}

There exists a point 
$Q_{\varepsilon}=  (\beta^o , \alpha^o- \varepsilon)$ close to $Q$ such that $E$ and $E(-rH)[1]$ are both tilt stable at $Q_{\varepsilon}$ with 
$$
\nu_{Q_{\varepsilon}}(E(-rH)[1]) < 0 < \nu_{Q_{\varepsilon}}(E).
$$
So we have 
\begin{equation*}
\Hom_{\SX}(E, E(-rH)[1]) = 0.
\end{equation*}
From the Serre duality, $\Hom_{\SX}(E, E[2]) = 0$.
Since $E$ is tilt stable, $\Hom_{\SX}(E, E) \cong \CC$, Therefore,
\begin{align*}
\chi(E, E) & = \sum_{i \in \ZZ} \hom_{\SX}(E, E[i]) \le \hom_{\SX}(E, E) +  \hom_{\SX}(E, E[2]) = \hom_{\SX}(E, E)  = 1.
\end{align*}
On the other hand from the Riemann-Roch formula \eqref{eqn:RiemannRoch}
\begin{align*}
\chi(E, E) = \int_X \ch(E) \ch(E^{\vee}) \td_X = -(rH/2)(\ch_1(E)^2 - 2\ch_0(E) \ch_2(E)) + \ch_0(E)^2.
\end{align*}
Therefore,
\begin{align*}
H \cdot \Delta(E) \ge\frac{2}{r}(\ch_0(E)^2 -1).
\end{align*}
Since $\Del_H (E) \ge H^3 H \cdot \Delta(E)$ and $\ch_0(E) \ge 2$, we have 
$$
\frac{\Del_H (E) }{(H^3\ch_0(E))^2} \ge \frac{2(\ch_0(E))^2 -1)}{r H^3 \ch_0(E)^2} = \frac{2}{rd} \left(1 - \frac{1}{\ch_0(E)^2}\right) \ge \frac{2}{rd} \left(1 - \frac{1}{2^2}\right) = \frac{3}{2rd}.
$$
But this is not possible as we have chosen $\frac{\Del_H (E) }{(H^3\ch_0(E))^2} < \kappa(X) \le \frac{3}{2rd}$. 
This is the required contradiction to complete the proof of Theorem \ref{thm:strong-BG-ineq}.
\end{proof}

\section{Bogomolov-Gieseker  Type Inequality for Fano 3-folds}
\label{sec:Fano3-A=0}
Let $X$ be a Fano 3-fold of index $r$. So 
$$
-K_X = rH
$$
 for some ample divisor class $H$. 
Let the degree of $X$ be 
$$
d = H^3.
$$
 We carry the same notation in Section \ref{sec:BGconj} for our Fano 3-fold $X$.
We only consider our modified Bogomolov-Gieseker type conjecture on $X$ when  $B$ is proportional to $H$.

From Proposition \ref{prop:Fano3-prop}, the Todd class of $X$ is
$$
\td(X) = \left(1, \frac{rH}{2}, \frac{r^2 H^2}{12} + \frac{c_2(X)}{12},1\right).
$$
\begin{prop}
\label{prop:Euler-char-Fano3}
Let $E \in D^b(X)$. Then for any $\beta \in \mathbb{R}$, we have 
$$
\chi(E(-H)) = \ch_3^{\beta H}(E) + {f_2(\beta) H} \ch_2^{\beta H}(E) + \Lambda \ch_1^{\beta H}(E) + f_1(\beta) H^2 \ch_1^{\beta H}(E) + f_0(\beta) H^3 \ch_0(E),
$$
where 
\begin{align*}
\Lambda & = \frac{c_2(X)}{12} - \frac{2}{rd}H^2, \\
f_2(\beta) & = \beta + \left( \frac{r}{2} -1\right),\\
f_1(\beta) & =\frac{\beta^2}{2}  + \left( \frac{r}{2} -1 \right) \beta  + \left( \frac{1}{2} - \frac{r}{2} + \frac{r^2}{12} + \frac{2}{rd} \right),\\
f_0(\beta) & =\frac{\beta^3}{6} + \left( \frac{r}{2} -1 \right)\frac{\beta^2}{2}  + \left( \frac{1}{2} - \frac{r}{2} + \frac{r^2}{12} + \frac{2}{rd} \right)\beta + \left(-\frac{1}{6} + \frac{r}{4} - \frac{r^2}{12} - \frac{2}{rd} + \frac{1}{d} \right).
\end{align*} 
\end{prop}
\begin{proof}
From the Hirzebruch-Riemann-Roch theorem \eqref{eqn:RiemannRoch}, we have 
$$
\chi(E(-H)) = \ch_3(E(-H)) +\frac{r H}{2} \ch_2(E(-H)) + \left(\frac{r^2 H^2}{12} + \frac{c_2(X)}{12}\right) \ch_1(E(-H)) + \ch_0(E(-H)).
$$
From Proposition \ref{prop:Fano3-prop},
 we have $(c_2(X) .H)/(12H^3) = 2/(rd)$.
Since $\ch(E(-H)) = \ch(E) \cdot e^{-H}$,  and 
  $\ch(E) = \ch^{\beta H}(E) \cdot e^{\beta H}$ one can get the required expression 
in terms of $ \ch^{\beta H}_i$'s.  
\end{proof}

\begin{note}
\label{prop:f1-f0-inequlaties}
\rm
Let us use the  Iskovskikh-Mori-Mukai classification 
of smooth Fano 3-folds, to find certain inequalities involving $f_0(\beta)$ and $f_1(\beta)$ in Proposition \ref{prop:Euler-char-Fano3}, when $\beta \in [0,1)$ according to the index $r$ of the Fano 3-fold  $X$. 
\begin{enumerate}
\item When  $r=4$, $X$ is $\PP^3$; so $d=H^3=1$.  Hence, $f_0(\beta) = (\beta^3/6) + (\beta^2/2) + (\beta/3) \ge 0$, and $f_1(\beta) = (\beta^2/2) + \beta + (1/3)>0$, for $\beta \in [0,1)$.

\item When  $r=3$, $X$ is a quadric; so $d=H^3=1$.  Hence, $f_0(\beta) = (\beta^3/6) + (\beta^2/4) + (5\beta/12) + (1/6) > 0$, and $f_1(\beta) = (\beta^2/2) + (\beta/2) + (5/12)>0$, for $\beta \in [0,1)$.

\item  When  $r=2$, $1 \le d \le 7$. 
By simplifying 
$f_0(\beta) = (\beta^3 + ((6-d)\beta /(6d))$, and $f_1= (\beta^2/2) + ((6-d)/(6d))$. 
Hence, when $1 \le d \le 6$, we have $f_0(\beta),  f_1(\beta) \ge 0$ for $\beta \in [0,1)$.

\item When  $r=1$, $1 \le d \le 62$. 
By simplifying 
\begin{align*}
f_0(\beta) & = \frac{1}{6}\left( \beta - \frac{1}{2} \right)^3 + \left(\frac{48-d}{24d}\right)\left( \beta -\frac{1}{2}\right), \\
f_1(\beta) & = \frac{1}{2}\left( \beta - \frac{1}{2} \right)^2 + \left(\frac{48-d}{24d}\right).
\end{align*}
\end{enumerate}
\end{note}

We have the following:
\begin{prop}
\label{prop:def-xi-A=0}
There exists a minimum $\xi \ge 0$ satisfying
\begin{enumerate}
\item $f_1(\beta) + \xi \ge 0$, 
\item $ (1-\beta)\left(f_1(\beta) + \xi \right) + 2 f_0(\beta) \ge 0 $, and 
\item $ \sqrt{\kappa(X)} \left(f_1(\beta) + \xi \right) + f_0(\beta) \ge 0$,
\end{enumerate}
for all $\beta \in [0,1]$.
Here $\kappa(X)$ is the constant as in Definition \ref{def:kappa(X)}.
\end{prop}
\begin{proof}
From Note \ref{prop:f1-f0-inequlaties}, for large enough $\xi$,  we have 
$f_1(\beta) + \xi \ge 0$, $ (1-\beta)\left(f_1(\beta) + \xi \right) + 2 f_0(\beta) \ge 0 $, and 
$ \sqrt{\kappa(X)} \left(f_1(\beta) + \xi \right) + f_0(\beta) \ge 0$, for all $\beta \in [0,1]$.
Therefore, there is a minimum $\xi$ satisfying those inequalities.
\end{proof}

We prove that Conjecture \ref{conj:limitBG} holds for $X$ with respect to the zero function $A=0$ and for some $0 \le \xi(A) \le \xi$. 

\begin{thm}
\label{thm:fano3-A=0}
Let $A$ be the zero function:
$A : \mathbb{R} \langle H \rangle \to \mathbb{R}_{\ge 0}$, $\beta H \mapsto 0$.
Let $E \in D^b(X)$ be a $\obe_{A}$-stable object.
 Then we have $D^{0,\xi}_{0,  \obe} (E) \le 0$ for $\obe \in \obe_{A}(E)$. 
\end{thm}

\begin{rmk}
\label{rem:xi=0-cases}
\rm
For the following cases, the constant $\xi (A)=0$:
\begin{itemize}
\item $r=4$
\item $r=3$
\item $r=2$ with $1 \le d \le 6$
\item From Proposition \ref{prop:function-bound-r=1-dle48}, Fano 3-folds $X$ with index $r=1$ having degree $1\le d \le 48$ and $\kappa(X) = 3/(2rd)$; in particular, 
 Fano 3-folds of index one with Picard rank one. 
\end{itemize}
\end{rmk}

\begin{note}
\label{prop:f1-strict-ineq}
By using the functions $f_0$ and $f_1$ in Note \ref{prop:f1-f0-inequlaties}, one can show that for $\xi$
 defined in Proposition \ref{prop:def-xi-A=0} 
$$
f_1(\beta) + \xi >0, \ \text{for } \beta \in [0,1].
$$
\end{note}

We adapt some of the techniques from \cite{Li, BMSZ} to prove  Theorem \ref{thm:fano3-A=0}.
First we need the following:
\begin{prop}
\label{prop:support-prop-for-A=0-case}
Let $E \in D^b(X)$ be a $\obe_{A}$-stable object, with
$\Del_{H}(E) >0$, 
$\ch_0(E) \ge 0$,  
$ \obe_A(E) \subset [0,1)$ and $\chi(E(-H)) \le 0$.
 Then we have $D^{0,\xi}_{0,  \obe} (E) \le 0$ for each $\obe \in \obe_A(E)$. 
\end{prop}
\begin{proof}
Let $\obe \in \obe_{A}(E)$.  
We have $H\ch_2^{\obe H}(E) = 0$, and so $\Del_H(E) = (H^2 \ch_1^{\obe H}(E))^2$. 
Since $\Del_H(E) >0$, $H^2 \ch_1^{\obe H}(E)>0$.

From Proposition \ref{prop:Euler-char-Fano3}, we have
$$
0 \ge \chi(E(-H)) = D^{0,\xi}_{0,  \obe}(E) + \left(f_1(\obe) + \xi\right) H^2\ch_1^{\obe H}(E) + f_0(\obe) H^3 \ch_0(E).
$$
Since $ H^2\ch_1^{\obe H}(E) \left(f_1(\obe) + \xi\right) \ge 0$, when $\ch_0(E) = 0$, we have $0 \ge \chi(E(-H)) \ge D^{0,\xi}_{0,  \obe}(E)$ as required. 
So let us assume $\ch_0(E)>0$. 

Now suppose 
$D^{0,\xi}_{0,  \obe}(E) >0$
for a contradiction. Therefore, we have $\left(f_1(\obe) + \xi\right) H^2\ch_1^{\obe H}(E) + f_0(\obe) H^3 \ch_0(E) <0$.
Hence 
\begin{equation*}
\left( 2 \mu_H(E) - \obe \right) -1  < \frac{-\left((1- \obe)(f_1(\obe) + \xi) +2 f_0(\obe)\right)}{(f_1(\obe) + \xi)} \le 0.
\end{equation*}
Here the last inequality follows from the definition of $\xi$ in Proposition \ref{prop:def-xi-A=0} together with 
Note \ref{prop:f1-strict-ineq}.
Since $\obe \in [0,1)$, there are no integers in the interval 
$$
(\obe , 2\mu_H - \obe] \subset (0,1). 
$$
Therefore, from Theorem \ref{thm:strong-BG-ineq}, we get 
$$
\frac{(H^2 \ch_1^{\obe H}(E))^2}{(H^3 \ch_0(E))^2} \ge \kappa(X).
$$
So we have 
\begin{align*}
\left(f_1(\obe) + \xi\right) H^2\ch_1^{\obe H}(E) & + f_0(\obe) H^3 \ch_0(E) \ge 
\left( \sqrt{ \kappa(X) } \left(f_1(\obe) + \xi\right)  + f_0(\obe) \right) H^3 \ch_0(E) \ge 0.
\end{align*}
Here the last inequality follows from the definition of $\xi$ in Proposition \ref{prop:def-xi-A=0}.
Hence, we have $0 \ge \chi(E(-H)) \ge D^{0,\xi}_{0,  \obe}(E)$. This is the required contradiction. 
\end{proof}

\begin{proof}[Proof of Theorem \ref{thm:fano3-A=0}]
 Let  $\obe \in \obe_{A}(E)$. 
 
If $\Del_H(E) =0$ then from 
Lemma  \ref{prop:Delta0tiltstableFano3} such objects are isomorphic to 
$\oO_X(mH)[1]$ or $\iI_Z(mH)$ for some  $m \in \ZZ$ and $0$-subscheme $Z \subset X$.
By direct computation one can check the Bogomolov-Gieseker type inequalities hold for such  objects. 
So we can assume 
$$
\Del_H(E) > 0.
$$ 
 
 By using Proposition \ref{prop:stability-B-dual}, and since tilt stability is preserved under the tensoring by a line bundle, 
further  we can assume 
 $$
 \ch_0(E) \ge 0 \ \text { and } \ \obe \in [0,1). 
 $$

  \subsection*{(i) When we do not have the case $r=1$ with $\obe =0$}
 From  Proposition \ref{prop:betabar-hom-vanishing}, for any $j \le 0$ we have
\begin{align*}
  \Hom_{\SX}(\oO_X(H), E[j]) =0, \  \text { and } \
   \Hom_{\SX}(E, \oO_X((1-r)H)[1+j]) =0.
\end{align*}
Since $\omega_X = \oO_X(-rH)$, from the Serre duality, $\Hom_{\SX}(E, \oO_X((1-r)H)[1+j]) \cong \Hom_{\SX}(\oO_X(H), E[2-j])^*$.
Therefore,  from the Hirzebruch-Riemann-Roch theorem
$$
\chi(\oO_X(H), E) = \sum_{i \in \mathbb{Z}} (-1)^i \hom_{\SX}(\oO_X(H), E[i]) = - \hom_{\SX}(\oO_X(H), E[1]) \le 0.
$$
That is $\chi(E(-H)) \le 0$. So from Proposition \ref{prop:support-prop-for-A=0-case}, we have the required inequality.

 \subsection*{(ii) When we have the case $r=1$ with $\obe =0$}
 Suppose for a contradiction there exists a counterexample $E$;  so $D^{0, \xi}_{0,0}(E) > 0$.
 Since $H \ch_2(E) =0$, 
 $\Del_H(E) = (H^2 \ch_1(E))^2 > 0$.
 Let $E$ be one such example having minimum
 $H^2\ch_1$.

 If $\chi(E(-H)) \le 0$ then from Proposition \ref{prop:support-prop-for-A=0-case} we have 
 $D^{0, \xi}_{0,0}(E) \le 0$; however, since we already assumed $D^{0, \xi}_{0,0}(E) > 0$,   we have 
$\chi(E(-H)) >0$. 

From  Proposition \ref{prop:betabar-hom-vanishing}, for any $j \le 0$ we have
\begin{align*}
  \Hom_{\SX}(\oO_X(H), E[j]) =0, \
   \Hom_{\SX}(E, \oO_X[j]) =0.
\end{align*}
By using    the Serre duality we get
$$
\chi(\oO_X(H), E) = \sum_{i \in \mathbb{Z}} (-1)^i \hom_{\SX}(\oO_X(H), E[i]) = - \hom_{\SX}(\oO_X(H), E[1]) +  \hom_{\SX}(\oO_X(H), E[2]).
$$

Since $\chi(E(-H)) > 0$, from the Serre duality we have $\hom_{\SX}(\oO_X(H), E[2]) \cong \hom_{\SX}(E, \oO_X[1]) \ne 0$. So there is a non-trivial map $E \to \oO_X[1]$ in $D^b(X)$ and hence, we have the distinguished triangle
$$
\oO_X \to E_1 \to E \to \oO_X[1]
$$
for some $E_1 \in D^b(X)$. 
Here $E, \oO_X[1] \in \bB_{0}$, and also from Proposition \ref{prop:minimalobj}, $ \oO_X[1]$ is a minimal object. Therefore by considering the long exact sequence of $ \bB_{0}$-cohomologies we get 
$E_1 \in  {\bB_{0}}$ and 
the following non-splitting short exact sequence in $ \bB_{0}$:
$$
0 \to E_1 \to E \to \oO_X[1] \to 0.
$$
By applying the functor $\Hom_{\SX}(-,\oO_X[1])$ to the above short exact sequence, we get
$$
\hom_{\SX}(E_1, \oO_X[1]) = \hom_{\SX}(E , \oO_X[1]) - \hom_{\SX}(\oO_X[1], \oO_X[1]) < \hom_{\SX}(E, \oO_X[1]).
$$
We have $\ch_0(E_1) =\ch_0(E) +1 >0$, and
$$
\chi(E_1(-H)) = \chi(E(-H)) - \chi(\oO_X(-H)[1]) = \chi(E(-H)) - 1.
$$
If $\chi(E_1(-H)) > 0$ then we can repeat the above process for $E_1$.
In this way we get a sequence of non-splitting short exact sequences 
$$
0 \to E_{i+1} \to E_{i} \to \oO_X[1] \to 0
$$
in $\bB_0$ for some $E_i \in \bB_0$, with $E_0 = E$, and 
$$
\cdots< \chi(E_i(-H)) < \cdots   < \chi(E_1(-H)) < \chi(E(-H)) .
$$
Therefore, for some $m \ge 1$ we have the short exact sequence 
$$
0 \to F:=E_m \to E \to \oO_X^{\oplus m}[1] \to 0
$$
in $\bB_0$ with 
$\chi(F(-H))\le 0$. Also
$\ch_0(F) =\ch_0(E) +m$, and for $i \ge 1$, $\ch_i(F) =\ch_i(E)$.

Let us prove $F \in \bB_0$ is $\obe_A$ stable with $\obe_A(F) =\{0\}$.
Consider $\nu_{0, \alpha}$ tilt stability of $F$ for sufficiently small $\alpha>0$.
 Since $\Del_H(E) >0$, $H^2 \ch_1(E) = H^2 \ch_1(F) >0$.
 Hence, from the properties of Harder-Narasimhan filtrations and Jordan-H\"older filtrations, 
 there is a filtration:
 $$
 0 = F_0 \subseteq F_1 \subseteq \cdots \subseteq F_k = F
 $$
with $G_i : = F_i/F_{i-1} \in \bB_0$ are $\nu_{0,\alpha}$ stable  with
$$
\nu_{0,\alpha}(G_i) \ge  \nu_{0,\alpha}(G_{i+1})
$$
for sufficiently small enough $\alpha >0$. 
Since $H^2 \ch_1(F) >0$,  $\nu_{0,\alpha}$ slope of each $G_i$'s are finite. 
So $H^2\ch_1(G_i) \ge 1$. 
Moreover, $G_1 = F_1$ is a subobject of $E$. So 
$$
\nu_{0,\alpha}(E) >  \nu_{0,\alpha}(F_1) \ge \cdots \ge \nu_{0,\alpha}(F_{k-1})   \ge \nu_{0,\alpha}(F_k)  = \nu_{0,\alpha}(F)
$$
for sufficiently small enough $\alpha >0$. 
Therefore, by considering the limit  $\alpha \to 0^+$, we get 
$H\ch_2(F_i) =0$ for all $i$. So 
$H \ch_2(G_i) = 0$ for all $i$. 
Since 
$0<D^{0,\xi}_{0,0}(E) = D^{0,\xi}_{0,0}(F) = \sum_i D^{0,\xi}_{0,0}(G_i)$,
there is at least one $G_i$ such that $D^{0,\xi}_{0,0}(G_i) >0$. 
Therefore $\obe_A$ stable object 
$G_i^1 : = H^{1}_{\bB_0}(G_i^\vee)$ also satisfies $D^{0,\xi}_{0,0} >0$.
Since $\ch_0(G_i) = - \ch_0(G_i^1)$, one of $G_i, G_i^1$ have $\ch_0 \ge 0$.
That is one of  $G_i, G_i^1$ is also a counterexample like $E$.
Since  we assumed $E$ to be a counterexample with minimal $H^2\ch_1$, we 
have $H^2\ch_1(G_i) = H^2 \ch_1(E)$. 
So $F$ is also $\obe_A$ stable with $\obe_A (F) = \{0\}$.

Since $\chi(F(-H)) <  0$, from Proposition \ref{prop:support-prop-for-A=0-case} we get $D^{0,\xi}_{0,0}(F)\le 0$; this is the required contradiction.
This completes the proof.
\end{proof}

\section{Optimal Bogomolov-Gieseker Type Inequality for Blow-up of $\mathbb{P}^3$ at a Point}
\label{sec:blowup-P3}
Let $X$ be the blow-up of $\mathbb{P}^3$ at a point. 
From   Iskovskikh-Mori-Mukai classification 
of smooth Fano 3-folds $X$ is the only Fano 3-fold of index $r =2$ having degree $d >6$. 
In particular, the degree of $X$ is $d=7$.
In this section we optimize the Bogomolov-Gieseker type inequality for $X$. More precisely, we show that the modified inequality in  Conjecture \ref{conj:BG-ineq}  holds with $\xi(A) =0$ for some $A \ne 0$.

\begin{defi}
\label{def:blowup-P3-A-for-x=0}
The  continuous function  
$A: \RR \langle H \rangle \to \RR_{\ge 0}$ 
is defined by
\begin{equation*}
\begin{cases}
\text{for } \beta \in [-1/2, 0), \  A(\beta H) = 1+ \beta;\\
\text{for } \beta \in [0, 1/2), \  A(\beta H) = 1 -  \beta;\\
\text{together with the relation } A((\beta +1)H) = A(\beta H).\\
\end{cases}
\end{equation*}
See Figure \ref{fig: graph-A-for-xi=0}.
\end{defi}

\begin{center}
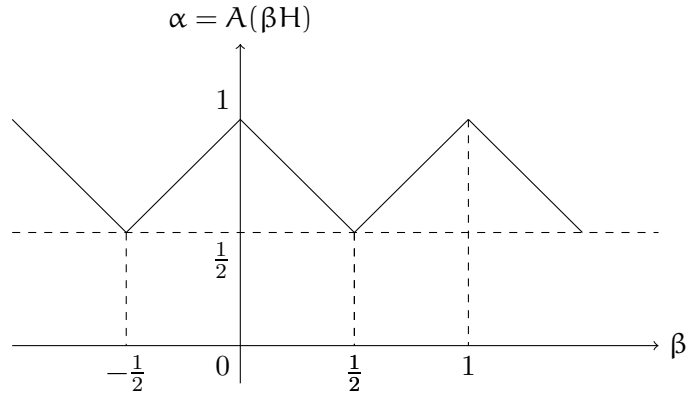

\begin{tikzpicture}
      \draw[->] (-3,0) -- (5.5,0) node[right] {$\beta$};
      \draw[->] (0,-0.5) -- (0,4) node[above] {$\alpha = A(\beta H)$};
            \draw[scale=3,domain=-1:-0.5,smooth,variable=\x,black] plot ({\x},{(-\x)});
            \draw[scale=3,domain=-0.5:0,smooth,variable=\x,black] plot ({\x},{(1+\x)});
      \draw[scale=3,domain=0:0.5,smooth,variable=\x,black] plot ({\x},{(1-\x)});
                  \draw[scale=3,domain=0.5:1.0,smooth,variable=\x,black] plot ({\x},{(\x)});    
                  \draw[scale=3,domain=1:1.5,smooth,variable=\x,black] plot ({\x},{(2- \x)});   
    \draw [dashed]  (-3,1.5)  -- (0,1.5) node [below left] {$\frac{1}{2}$}
       -- (5.5,1.5);
    \draw [dashed]  (1.5,1.5)  -- (1.5,0) node [below] {$\frac{1}{2}$};
        \draw [dashed]  (3,3)  -- (3,0) node [below] {$1$};
            \draw [dashed]  (1.5,1.5)  -- (1.5,0) node [below] {$\frac{1}{2}$};
                    \draw [dashed]  (-1.5,1.5)  -- (-1.5,0) node [below] {$-\frac{1}{2}$};
            \draw  (0,0)  node [below left] {$0$};
              \draw  (0,3)  node [above left] {$1$};
    \end{tikzpicture}
        \captionof{figure}{Graph of $\alpha = A(\beta H)$ in $(\beta, \alpha)$-plane}
        \label{fig: graph-A-for-xi=0}
\end{center}

We prove Conjecture \ref{conj:BG-ineq} or equivalently Conjecture \ref{conj:limitBG} holds for $X$ with respect to the function $A$ and constant 
$$
\xi(A)=0.
$$

\begin{thm}
\label{thm:fano3-d-7}
Let $E \in D^b(X)$ be a $\obe_{A}$-stable object. Then 
$D^{0, 0}_{A(\obe H),  \obe} (E)\le 0$ for each $\obe \in \obe_A(E)$. 
In particular, Conjecture \ref{conj:compareAforBGineq} holds for $X$. 
\end{thm}
We need the following:
\begin{prop}
\label{prop:support-prop-blowup-P3-main-thm-xi=0}
Let $E \in D^b(X)$ be a $\obe_{A}$-stable object with $\Del_H(E) >0$, $\ch_0(E)\ge 0$,  
$\obe_{A}(E) \subset [-1/2, 1/2)$, and $\chi(E(-H)) \le 0$. Then we have
$D^{0, 0}_{A(\obe H),  \obe} (E)\le 0$ for each $\obe \in \obe_A(E)$.
\end{prop}
\begin{proof}
Let  $\obe  \in  \obe_{A }(E)$ and $\oA =  A(\obe H)$.
So $H \ch_2^{\obe H}(E) = (\oA^2/2)H^3 \ch_0(E)$, and 
from  Proposition \ref{prop:Euler-char-Fano3},
\begin{align*}
\chi(E(-H)) =   D^{0,0}_{A, \obe }(E) + p  H^2 \ch_1^{ \obe H} (E) + q H^3 \ch_0(E), 
\end{align*}
 where 
 \begin{align*}
p  & = \frac{1}{6}\oA^2+  \frac{1}{2} \obe^2  - \frac{1}{42}, \\
q &= \obe \left(    \frac{1}{2}\oA^2+  \frac{1}{6} \obe^2  - \frac{1}{42} \right).
 \end{align*}
By using the definition of $A$ and simplifying $p$ we obtain, 
for $\obe \in [-1/2, 0)$, $p = (1/24)(4 \obe +1)^2 + (17/168)$, and for 
$\obe \in [0, 1/2) $, $p= (1/24)(4 \obe - 1)^2 + (17/168)$. Therefore, 
$$
p>0.
$$
Hence, since  $H^2 \ch_1^{ \obe H} (E) \ge 0$, when $\ch_0(E) =0$ we have $0 \ge \chi(E(-H)) \ge D^{0,0}_{\oA, \obe}(E)$ as required.

So we assume $\ch_0(E)>0$.

\subsection*{(i) Case  $\obe \in [0, 1/2)$:}
We have $\oA = 1-\obe$ and so 
\begin{align*}
q = \obe \left(\frac{(1-\obe)^2}{2}  + \frac{\obe^2}{6} - \frac{1}{42}  \right) 
= \frac{2}{3}\obe \left( \left(\obe - \frac{3}{4}\right)^2 + \frac{17}{112} \right) \ge 0.
\end{align*}
Since $H^2 \ch_1^{ \obe H} (E) , H^3\ch_0(E) \ge 0$, we have $0 \ge \chi(E(-H)) \ge D^{0,0}_{\oA, \obe}(E)$ as required.

\subsection*{(ii) Case  $\obe \in [-1/2, 0)$:}
We have $\oA = 1+\obe$.

Suppose for a contradiction $E$  is a counterexample; that is
$D^{0, 0}_{\oA,  \obe} (E)> 0$.
Since  $\chi(E(-H)) \le 0$ and $D^{0, 0}_{\oA,  \obe} (E)> 0$, we have 
$ p  H^2 \ch_1^{ \obe H} (E) + q H^3 \ch_0(E) <0$. Therefore,
\begin{align*}
\left(2 \mu_H(E) - \obe -\oA \right)  - 1& < \frac{-2p -2q}{p} <0.
\end{align*}
Here the last inequality follows from (i) of Proposition \ref{prop:fun-bound-blowup-P3-xi=0-case}.
On the other hand $\obe + \oA = 1+2 \obe \in [0,1)$. Therefore, there are no integers in the interval
$$
(\obe + \oA , \ 2 \mu_H(E) - \obe -\oA] \subset (0,1).
$$
Hence, from Theorem \ref{thm:strong-BG-ineq}, we have 
$$
\frac{\Del_H(E)}{(H^3 \ch_0(E))^2} \ge \kappa(X) = \frac{2}{49}.
$$
Therefore,
$$
\frac{H^2 \ch_1^{\obe H}(E)}{H^3 \ch_0(E)} \ge \sqrt{\oA^2 + \kappa(X)}.
$$
So
$$
p  H^2 \ch_1^{ \obe H} (E) + q H^3 \ch_0(E) \ge \left( p \sqrt{\oA^2 + \kappa(X)} + q \right) H^3 \ch_0(E) \ge 0.
$$
Here the last inequality follows from (ii) of Proposition \ref{prop:fun-bound-blowup-P3-xi=0-case}.
Therefore, $0 \ge \chi(E(-H)) \ge D^{0,0}_{\oA, \obe}(E)$; this is the required contradiction.

This completes the proof.
\end{proof}

\begin{proof}[Proof of Theorem \ref{thm:fano3-d-7}]
 Let  $\obe \in  \obe_{A}(E)$ and $\oA =  A(\obe H)$. 

If $\Del_H(E) =0$ then from 
Lemma \ref{prop:Delta0tiltstableFano3} such objects are isomorphic to 
$\oO_X(mH)[1]$ or $\iI_Z(mH)$ for some  $m \in \ZZ$ and $0$-subscheme $Z \subset X$.
By direct computation one can check the Bogomolov-Gieseker  type inequalities hold for such  objects. 
So we can assume 
$$
\Del_H(E) > 0.
$$ 
 
Since tilt stability is preserved under tensoring by a line bundle, and 
also from Proposition \ref{prop:stability-B-dual} we further can assume
 $$
 \ch_0(E) \ge 0, \ \text{ and } \  \obe \in [-1/2, 1/2). 
 $$

From  Proposition  \ref{prop:betabar-hom-vanishing},  for any $j \le 0$ we have 
$$
\Hom_{\SX}(\oO_X(H), E[j]) =0, \ \ \text{and} \ \ \Hom_{\SX}(E, \oO_X(-H)[1+ j]) =0.
$$
By the Serre duality, $\Hom_{\SX}(E, \oO_X(-H)[1+j]) \cong \Hom_{\SX}(\oO_X(H), E[2-j])^*$.
Hence, 
$$
\chi(\oO_X(H), E) = \sum_{i \in \mathbb{Z}} (-1)^i \hom_{\SX}(\oO_X(H), E[i]) = - \hom_{\SX}(\oO_X(H), E[1]) \le 0.
$$
From Proposition \ref{prop:support-prop-blowup-P3-main-thm-xi=0} we have the required inequality.
\end{proof}

\appendix
\counterwithin{thm}{section}
\section{}

\begin{prop}
\label{prop:function-bound-r=1-dle48}
Let $d$ be a positive integer such that $1 \le d \le 48$.
Let $f: \RR \to \RR$ be the   function defined by 
$$
f(x) = \frac{1}{6}\left(x-\frac{1}{2}\right)^3 + \frac{(48-d)}{24d} \left(x -\frac{1}{2}\right).
$$
Let $f'(x)$ be the derivative of $f(x)$ with respect to $x$.
Then for $x \in [0,1)$ we have 
\begin{enumerate}
\item $\sqrt{{3}/(2d)} \, f'(x) + f(x) \ge 0$.
\item $(1-x)f'(x) + 2f(x)\ge 0$.
\end{enumerate}
\end{prop}
\begin{proof}
(i) \
Let $g$ be the  function defined by,  
$g(x) = \sqrt{{3}/(2d)} \, f'(x) + f(x)$.
Here  $f'(x) = (1/2)(x-(1/2))^2  + (48-d)/(24d)$.

By differentiating $g(x)$ we get
\begin{align*}
g'(x) = & \frac{1}{2}\left(x-\frac{1}{2}\right)^2 +  \sqrt{\frac{3}{2d}}\left(x-\frac{1}{2}\right) + \frac{(48-d)}{24d} \\
= & \frac{1}{2} \left(x-\frac{1}{2} + \sqrt{\frac{3}{2d}} \right)^2 - \frac{(d-30)}{12d}.
\end{align*}
By evaluating $g(x)$ at $x=0$:
$$
g(0) = \sqrt{\frac{6}{d}}\left(\frac{1}{\sqrt{d}} - \frac{1}{\sqrt{24}}\right)^2 \ge 0.
$$
When $d \le 30$ we have $g'(x) \ge 0$ for all $x \in \mathbb{R}$ with $g(0)>0$. Hence for $x \in [0,1)$, $g(x) \ge 0$.

Let us consider the  case $d> 30$. The derivative $g'(x)$ is vanishing at $x = \lambda_1, \lambda_2$ with $\lambda_1 < \lambda_2$. Here 
\begin{align*}
\lambda_2 =  \frac{1}{2} - \sqrt{\frac{3}{2d}} + \sqrt{\frac{(d-30)}{12d}} = \frac{1}{2} - \frac{(48-d)}{\sqrt{12d}(\sqrt{18}+\sqrt{d-30})}.
\end{align*}
One can rearrange $g$ as 
$$
g(x) = \frac{1}{6}\left(x-\frac{1}{2} +\sqrt{\frac{3}{2d}} \right)^3 +  \frac{(30-d)}{24d} \left(x -\frac{1}{2}\right) + \sqrt{\frac{3}{2d}}\frac{(42-d)}{24d}.
$$
Let us consider the remaining case $30 < d \le 48$. The local minimum value of $g(x)$ at $x = \lambda_2$ is
\begin{align*}
g(\lambda_2) = - \frac{1}{3} \left(\frac{d-30}{12d} \right)^{3/2} + \frac{1}{2d} \sqrt{\frac{3}{2d}} =
\frac{18^{3/2} - (d-30)^{3/2} }{72\sqrt{3}d^{3/2}} \ge 0, 
\end{align*}
with equality when $d=48$. Moreover, $\lambda_2  = 1/2$ when $d=48$.  
Since $g(0) \ge 0$, and so for $x \in [0,1)$ we have $g(x) \ge 0$. \\

\noindent (ii) \ Let $h(x) =  (1-x)f'(x) + 2f(x)$. By simplifying we get
$$
h(x) = \frac{1}{24}(2-x)(1-2x)^2 + \frac{2(48-d)}{48d}.
$$
Since $1 \le d \le 48$, for all $x \in [0,1)$ we have $h(x) \ge 0$ as required.
\end{proof}

\begin{prop}
\label{prop:fun-bound-blowup-P3-xi=0-case}
Define the functions $p(x), q(x)$  by 
\begin{align*}
 p(x)  & = \frac{1}{6}(1+x)^2+  \frac{1}{2} x^2  - \frac{1}{42}, \\
q(x) &= x \left(\frac{1}{2}(1+x)^2 +  \frac{1}{6} x^2  - \frac{1}{42} \right).
 \end{align*}
For $x \in [-1/2,0]$,  we have the following:
\begin{enumerate}
\item $p(x) + q(x) > 0$.
\item $p(x) \, \sqrt{\left(1+x \right)^2 + \frac{2}{49}} + q(x) > 0$.
\end{enumerate}
\end{prop}
\begin{proof}
(i) \ By simplifying, we have 
$$
p(x) + q(x) = \frac{2}{3}x^3 + \frac{5}{3}x^2 + \frac{17}{21}x + \frac{1}{7}.
$$
Let us find the critical points of $p(x) + q(x)$.
The derivative of $p(x) + q(x)$ with respect to $x$ is $2(x+(5/6))^2- (73/126)$, and so its roots
are $-( 5/6) \pm \sqrt{73/252}$.
Therefore, the critical values of $p(x) +q(x)$ are 
$(1904 \mp 73 \sqrt{511})/7938 > 0$.
Hence, since 
$-( 5/6) - \sqrt{73/252} <-1/2< -( 5/6) + \sqrt{73/252} <0$, we get the required inequality for all $x \in[ -1/2,0]$. \\

\noindent (ii) \ For $x \in  [-1/2,0]$,
\begin{align*}
p(x) &= \frac{1}{24}(4 x +1)^2 + \frac{17}{168} >0, \text{ and}\\
q(x) & =  x \left( \frac{1}{24}(4 x +3)^2 + \frac{17}{168} \right) <0.
\end{align*}
So it is enough to show that for  $x \in  [-1/2, 0]$,
$$
F(x) :=\left( (1+x)^2 + \frac{2}{49}\right) p(x)^2 - q(x)^2 > 0.
$$
Let us find the critical points of $F$. By differentiating $F$ we get 
\begin{equation}
F'(x): = \frac{dF}{dx} = \frac{4}{3087} \left( 56x^3 + 483 x^2 + 460 x + 108 \right).
\end{equation}
The equation $F'(x)=0$ has three different real roots; so $F$ has three critical points. 
They are 
\begin{align*}
\lambda_1 & =- \frac{23}{8}+  \frac{1}{8}\left( \frac{\zeta}{\sqrt[3]{441}} + \frac{7429}{\sqrt[3]{21}\zeta} \right), \\
\lambda_2 & =- \frac{23}{8} -   \frac{(1+ i \sqrt{3})\zeta}{16 \sqrt[3]{441}} - \frac{7429(1- i \sqrt{3})}{16 \zeta}, \\
\lambda_3 & =- \frac{23}{8} -  \frac{(1-i \sqrt{3}) \zeta }{16 \sqrt[3]{441}} - \frac{7429(1+ i \sqrt{3})}{16 \zeta}.
\end{align*}
Here 
$$
\zeta = \left(-2917215 + 32 i \sqrt{97656006} \right)^{1/3} \in \RR_{>0} e^{(0, \ \pi/3)}.
$$
The critical points satisfy
$$
\lambda_3 < \lambda_2 <-1/2 <\lambda_1 <    0.
$$
Since $F$ is a degree $4$ polynomial of $x$ with a positive leading coefficient, we have for all $x \in  [-1/2, 0]$,
$$
F(x) \ge F(\lambda_1).
$$
On the other hand,  by direct computation we have
\begin{align*}
F(\lambda_1) 
& = -\frac{164676683}{33191424} 
+ \frac{3045940625581681}{4741632 \sqrt[3]{21} \, \zeta^4} 
- \frac{410006814589}{2370816 \sqrt[3]{441} \, \zeta^2}
+ \frac{i \sqrt{2325143} \, \zeta}{1555848\sqrt{2} \sqrt[6]{21}} \\
& \qquad \qquad - \frac{7429 \, \zeta^2}{49787136 \sqrt[3]{21}} 
+ \frac{37145 \sqrt[3]{9}}{64 \sqrt[3]{7}\, \zeta}
+\frac{5\sqrt[3]{21} \,\zeta}{512}
 >0.
\end{align*}
In fact, by direct computation one can check that
$$
0.000028<F(\lambda_1)<0.000029.
$$
This completes the proof.
\end{proof}

\providecommand{\bysame}{\leavevmode\hbox to3em{\hrulefill}\thinspace}
\providecommand{\MR}{\relax\ifhmode\unskip\space\fi MR }
\providecommand{\MRhref}[2]{%
  \href{http://www.ams.org/mathscinet-getitem?mr=#1}{#2}
}
\providecommand{\href}[2]{#2}

\end{document}